\documentclass[11pt]{article}

\usepackage[utf8]{inputenc}
\usepackage[T1]{fontenc}

\usepackage{amsmath,amsfonts,amsthm,amssymb}
\usepackage{bm,bbm}
\usepackage{enumerate}
\usepackage{stmaryrd}
\usepackage{color}
\usepackage{tikz}
\usepackage{float}
\usepackage{graphicx}
\usepackage{subfigure}
\usetikzlibrary{decorations.pathreplacing}

\theoremstyle{plain}
\newtheorem{proposition}{Proposition}[section]
\newtheorem{lemma}[proposition]{Lemma}
\newtheorem{theorem}[proposition]{Theorem}
\newtheorem{corollary}[proposition]{Corollary}

\theoremstyle{definition}
\newtheorem{definition}[proposition]{Definition}
\newtheorem{remark}[proposition]{Remark}
\newtheorem{example}[proposition]{Example}

\renewenvironment{thebibliography}[1]{%
\begin{oldthebibliography}{#1}%
\setlength{\baselineskip}{.9em}
\linespread{.95}
\small
\setlength{\parskip}{0.3ex}%
\setlength{\itemsep}{.35em}%
}%
{%
\end{oldthebibliography}%
}

\newcommand{\R}{\mathbb{R}}

\newcommand{\N}{\mathbb{N}}

\newcommand{\1}{\mathbf{1}}

\newcommand{\G}{\mathbb{G}}

\newcommand{\cI}{\mathcal{I}}
\newcommand{\cE}{\mathcal{E}}

\newcommand{\cG}{\mathcal{G}}
\newcommand{\cT}{\mathcal{T}}
\newcommand{\cU}{\mathcal{U}}
\newcommand{\cK}{\mathcal{K}}

\newcommand{\eps}{\varepsilon}
\newcommand{\smsp}{\,}
\DeclareMathOperator{\Law}{Law}
\DeclareMathOperator{\Unif}{Unif}
\DeclareMathOperator{\Beta}{Beta}

\newcommand{\kn}{\tfrac{k}{n}}
\newcommand{\konen}{\tfrac{k-1}{n}}

\numberwithin{equation}{section}

\usepackage[pdfborder={0 0 0}]{hyperref}
\hypersetup{
  urlcolor = black,
  pdfauthor = {Marcel Nutz, Jaime San Martin, Xiaowei Tan},
  pdfkeywords = {Mean Field Game; n-Player Game; Optimal Stopping},
  pdftitle = {Convergence to the Mean Field Game Limit: A Case Study},
  pdfsubject = {Convergence to the Mean Field Game Limit: A Case Study},
  pdfpagemode = UseNone
}

\begin{document}

\title{\vspace{-3.5em}Convergence to the Mean Field Game Limit:\\A Case Study}

\author{Marcel Nutz\thanks{
  Departments of Statistics and Mathematics, Columbia University, mnutz@columbia.edu. Research supported by an Alfred P.\ Sloan Fellowship and NSF Grants DMS-1512900 and DMS-1812661. MN is grateful to Daniel Lacker and Yuchong Zhang for numerous discussions which have greatly improved this paper.
  } 
  \and 
  Jaime San Martin\thanks{CMM-DIM, UMI 2807 CNRS, Universidad de Chile, jsanmart@dim.uchile.cl. Partially funded by BASAL AFB170001. JSM is thankful for the hospitality of the Statistics Department at Columbia University.
  }
  \and 
  Xiaowei Tan\thanks{Department of Mathematics, Columbia University, xt2161@columbia.edu.
  }}

\maketitle \vspace{-1em}

\begin{abstract}
We study the convergence of Nash equilibria in a game of optimal stopping. If the associated mean field game has a unique equilibrium, any sequence of $n$-player equilibria converges to it as $n\to\infty$.
However, both the finite and infinite player versions of the game often admit multiple equilibria. We show that  mean field equilibria satisfying a transversality condition are limit points of $n$-player equilibria, but we also exhibit a remarkable class of mean field equilibria that are not limits, thus questioning their interpretation as ``large $n$'' equilibria.
\end{abstract}

\vspace{.5em}

{\small

}

\section{Introduction}
\label{sec:Introduction}

Mean field games were introduced by \cite{LasryLions.06a,LasryLions.06b,LasryLions.07} and \cite{HuangMalhameCaines.07, HuangMalhameCaines.06} to overcome the notorious intractability of $n$-player games. Two key simplifications are made. First, agents interact symmetrically through the empirical distribution of their states. Second, by formally letting $n\to\infty$, one passes to a representative agent whose actions do not affect this distribution because each individual agent becomes negligible. Thus, the mean field game is seen as an approximation of the $n$-player game for large $n$. We refer to the lecture notes~\cite{Cardaliaguet.13} and the monographs~\cite{BensoussanFrehseYam.13,CarmonaDelaRue.17a,CarmonaDelaRue.17b} and their extensive references for further background.

In this paper, we conduct a case study of an $n$-player game of optimal stopping where multiple equilibria may occur naturally. We formulate an associated mean field game and highlight that certain mean field equilibria are limits of $n$-player equilibria while others are not, and study how to distinguish them. Equilibria that are not limit points are questionable from the point of view of applications, at least if they are motivated as ``$n$-player games with large $n$.'' 

Several ways of connecting $n$-player and mean field games have been studied in the literature. In many cases it is easier to establish the reverse direction, namely that a given mean field equilibrium induces an \emph{approximate} Nash equilibrium in the $n$-player game for large $n$. This goes back to \cite{HuangMalhameCaines.06} and is by now established in some generality, see in particular \cite{Lacker.14} for diffusion control,  \cite{CarmonaDelarueLacker.17} for games of timing or~\cite{CecchinFischer.18} for finite state games (but see also \cite{CampiFischer.18} for a counterexample in a degenerate case with absorption). It then follows, conversely, that mean field equilibria are limits of approximate $n$-player equilibria. However, we emphasize that approximate and actual Nash equilibria may look quite different, and in particular one cannot expect in general that there is a true Nash equilibrium in the proximity of an approximate one.

The convergence of $n$-player Nash equilibria to the mean field limit is often more delicate.
The  deep result of \cite{CardaliaguetDelarueLasryLions.15} shows convergence for a class of (closed-loop) games where agents choose drifts of diffusions. In their setting, the mean field game has a unique equilibrium as a consequence of the so-called monotonicity condition~\cite{LasryLions.06a} which postulates that it is disadvantageous for agents' states to be close to one another. In a related but different (open-loop) framework, and without imposing uniqueness, \cite{Fischer.14} obtains convergence under the assumption that the limiting measure flow is deterministic. More comprehensively, \cite{Lacker.14} shows that $n$-player equilibria converge to a weak notion of mean field equilibria which can include mixtures of deterministic equilibria, for a general class of diffusion-control games. A corresponding result for games of timing is established in \cite{CarmonaDelarueLacker.17}. Most recently, \cite{Lacker.18b} provides results along the lines of~\cite{Lacker.14} for the closed-loop case.
Convergence has also been shown in a number of more specific problems, for instance stationary mean field games~\cite{LasryLions.06a}, linear-quadratic problems~\cite{Bardi.12} or a game of Poissonian control~\cite{NutzZhang.17}, among others. However, to the best of our knowledge, the question which mean field equilibria are limit points of (true) $n$-player equilibria has not been emphasized as such in the literature. We can mention the parallel work~\cite{CecchinDaiPraFischerPelino.18} on a two-state game: the game has unique $n$-player equilibria and these converge to a mean field equilibrium as expected; however, a second, less plausible mean field solution can appear for certain parameter values and this solution is not a limit. Another interesting parallel work \cite{DelarueFoguenTchuendom.18} studies several approaches of selecting an equilibrium in a linear-quadratic mean field game with multiple equilibria, including the convergence of $n$-player equilibria. Different approaches are shown to select different equilibria.

From the perspective of mean field games, being a limit point of $n$-player equilibria can be seen as a stability property of equilibria with respect to the number of players. We are not aware of a systematic study in this direction (but see~\cite{BrianiCardaliaguet.18} for a recent investigation of a different stability property that is potentially related).  Since mean field equilibria are often motivated as ``large~$n$'' equilibria, it seems desirable to understand the phenomenon in some generality and at least establish sufficient conditions. A general formulation and investigation of this stability seems wide open at this time, whence our focus on a case study in the present paper.


\subsection{Synopsis}

We start by introducing an $n$-player game of optimal stopping inspired by~\cite{Bertucci.17,BouveretDumitrescuTankov.18,CarmonaDelarueLacker.17,Nutz.16} and the literature on bank-runs following~\cite{DiamondDybvig.83}. In addition to their i.i.d.\ signals, players observe how many other players have already stopped. A crucial feature is that whenever an agent leaves the game, staying in the game becomes less attractive for the remaining agents. For instance, this may reflect that the bank is more likely to default if other clients withdraw their savings. In particular, the game satisfies the opposite of Lasry and Lions' monotonicity condition, or \emph{strategic complementarity} in Economics terminology~\cite{BulowGeanakoplosKlemperer.85}. Indeed, the model exhibits a ``flocking'' or ``herding'' behavior where groups of agents can collectively decide to stop or not. We will see that these choices can naturally give rise to multiple equilibria; more precisely, they parametrize the full range of $n$-player equilibria.

Next, we review the mean field version of the game which was introduced in~\cite{Nutz.16} without discussing the $n$-player game. Enhancing slightly a result of~\cite{Nutz.16}, mean field equilibria are described by a simple equation: for any equilibrium, the proportion $\rho(t)$ of agents that have stopped by time~$t$ is a zero of a deterministic function $g_{t}$ on $[0,1]$ as is Figure~\ref{fig:intro}. More generally, any equilibrium $t\mapsto \rho(t)$ is characterized as an increasing, right-continuous selection of such zeros. In Figure~\ref{fig:intro}, we can distinguish several types of zeros: increasing-transversal~($i$), tangential~($t$) and decreasing-transversal~($d$).
\begin{figure}[b]
\begin{center}
\begin{tikzpicture}[scale=1.1]
\draw[line width=1pt,-latex] (-5,0) to (5,0) node[right=2pt] {$u$};
\fill (-4.25,0) circle (1.5pt) node[above=1pt] at (-4.25,0) {0};
\fill (4.25,0) circle (1.5pt) node[above=1pt] at (4.25,0) {1};
\draw[line width=1.5pt] (-4.25,-.5) 
to[out=0,in=225] (-3.25,0) 
to[out=45,in=135] (-1.75,0) 
to[out=-55,in=180] (-.575,0) 
to[out=0,in=235] (.5,0) 
to[out=45,in=135] (2,0)
to[out=-45,in=225] (3.5,0)
to[out=45,in=190] (4.25,.5)
node[above=1pt] {$g_{t}(u)$};
\fill (-3.25,0) circle (2pt) node[above=2pt]{$i$};
\fill (-1.75,0) circle (2pt) node[above=2pt]{$d$};
\fill (-.575,0) circle (2pt) node[above=2pt]{$t$};
\fill (.5,0) circle (2pt) node[above=2pt]{$i$};
\fill (2,0) circle (2pt) node[above=2pt]{$d$};
\fill (3.5,0) circle (2pt) node[above=2pt]{$i$};
\end{tikzpicture}
\end{center}
\caption{Types of mean field equilibria at a fixed time $t$}
\label{fig:intro}
\end{figure}
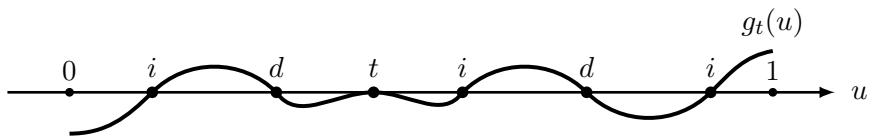
These types are related to how concentrated the distribution of the agents' signals is in a neighborhood of the zero, relative to the strength of interaction. Intuitively, tangential solutions are delicate in that they may disappear if Figure~\ref{fig:intro} is perturbed, whereas the transversal solutions are stable in this sense.

We then turn to our main question and study which mean field equilibria are limits of $n$-player equilibria. Roughly, the main result is that
\begin{enumerate}
\item Increasing-transversal solutions are limits of $n$-player equilibria,
\item decreasing-transversal solutions \emph{fail} to be limits,
\item tangential solutions can but need not be limits.
\end{enumerate}

Specifically, we first consider the minimal and maximal equilibria, corresponding to the left- and right-most solutions in Figure~\ref{fig:intro}. The $n$-player game also has such extremal equilibria and these yield natural candidates for sequences converging to their mean field counterparts. After introducing appropriate notions for dynamic equilibria, we show that this convergence indeed holds, under the condition that the solutions are increasing-transversal (on a sufficiently large set of times $t$). However, we also find that if the minimal (say) solution is tangential, the minimal $n$-player equilibria can converge to a mixture of mean field equilibria and then the minimal mean field equilibrium may fail to be a proper limit. (The minimal and maximal solutions can be increasing-transversal or tangential, but never decreasing-transversal.) This also yields a novel example of how randomization can emerge in mean field games.

Second, we study the convergence to a general mean field equilibrium, possibly somewhere in the middle of Figure~\ref{fig:intro}. In that case, there are no obvious candidates for the $n$-player approximations and more abstract arguments need to be used. We show by a fixed point construction that all increasing-transversal solutions are limits of $n$-player equilibria. Quite surprisingly however, (``strongly'') decreasing-transversal solutions fail to be limits despite appearing stable in Figure~\ref{fig:intro}. In fact, these solutions merely occur as parts of mixtures that are limits, and the weight within these mixtures can be bounded by a monotone function of the slope in Figure~\ref{fig:intro}. It turns out that some fairly detailed asymptotic statistics, such as the expected number of $n$-player equilibria, can be analyzed in our model---which is unusual for mean field games.

The remainder of this paper is organized as follows. In Section~\ref{se:basicSetup}, we introduce the game of optimal stopping. Section~\ref{se:nPlayer} describes the Nash equilibria of the $n$-player version and Section~\ref{se:MFG} covers the analogue for the mean field game. The results on the convergence to the minimal and maximal equilibria are relatively direct and established in Section~\ref{se:convergenceExtremal}, whereas the more abstract results on the convergence to general equilibria are reported in Section~\ref{se:convergenceGeneral}.


\section{Description of the Game}\label{se:basicSetup}

Let $(I,\cI,\lambda)$ be a probability space representing the agents; we shall be interested in the $n$-player case with a finite $I$ and the mean field case with an atomless space. Let $(\Omega,\cG,P)$ be another probability space, equipped with a right-continuous filtration $\G=(\cG_{t})_{t\in\R_{+}}$ and an exponentially distributed random variable $\cE$ which is independent of $\G$.

Given an agent $i\in I$, let $\alpha^{i}\geq0$ be a $\G$-progressively measurable process which is locally integrable  
and consider the random time
$$
  \theta^{i} = \inf\bigg\{t:\, \int_{0}^{t} \alpha^{i}_{s}\,ds = \cE\bigg\}.
$$
As in~\cite{Nutz.16}, one may think of $\theta^{i}$ as the time when agent $i$ expects the default of her bank.
We fix a parameter $r\in\R$, interpreted as the interest rate paid by the bank (and assumed to be constant for simplicity). Following~\cite{Nutz.16}, we suppose that $\alpha^{i}$ is increasing\footnote{Increase is to be understood in the non-strict sense throughout the paper.} and that 
\begin{equation}\label{eq:intCond}
    \text{ $\inf\{t: \, \alpha^{i}_{t}-r\geq 0\}<\infty$\quad $P$-a.s.} \;
\end{equation}
Denoting by $\cT$ the set of all $\G$-stopping times, we then consider the 
optimal stopping problem
  \begin{equation}\label{eq:optStopProblem}
    \sup_{\tau\in\cT} E\big[e^{r\tau} \1_{\{\theta^{i}>\tau\}\cup \{\theta^{i}=\infty\}}\big]
  \end{equation}
which we assume to have a finite value. Thus, if the default $\theta^{i}>\tau$, we may think of the agent as accruing the interest on an initial unit investment until~$\tau$, but losing everything if $\theta^{i}<\tau$.
If the stopping time 
\begin{equation}\label{eq:optStopTime}   
   \tau^{i} := \inf\{t: \, \alpha^{i}_{t}\geq r\} \in\cT
\end{equation}
is a.s.\ finite, then $\tau^{i}$ is optimal and in fact the minimal solution of~\eqref{eq:optStopProblem}; cf. \cite[Lemma~2.1]{Nutz.16}.
The solution is unique for instance if $\alpha^{i}$ is strictly increasing, but not in general. We assume that agents choose~\eqref{eq:optStopTime} in the case of non-uniqueness, which can be motivated e.g.\ as a preference for early stopping when other things are equal. This convention is not essential, but simplifies our exposition and allows us to focus on multiplicity of equilibria due to inherent game-theoretic aspects as it avoids ambiguity at the individual agents' level.

The processes $\alpha^{i}$ will depend on the proportion $\rho(t)$ of players who have already stopped, thus inducing an interaction among the agents. Since given~$\rho$, the optimal stopping times are completely determined by~\eqref{eq:optStopTime}, we shall simply say that an equilibrium is a process $\rho$ which is $\G$-adapted and such that
\begin{equation*}\label{eq:equilibDef}
  \rho(t) = \lambda\{i:\, \tau^{i}\leq t\},
\end{equation*}
where it is tacitly assumed that the above set is $\lambda$-measurable.

\section{The $n$-Player Game}
\label{se:nPlayer}

In this section, we formulate the $n$-player version of the ``toy model'' mean field game in \cite[Section~4]{Nutz.16}. Indeed, fix $n\in\N$ and take $I=\{1,\dots,n\}$ to be a set with $n$ elements, equipped with the normalized counting measure.  Each player $i$ observes an idiosyncratic signal  $Y^i_t\geq 0$ which is right-continuous, progressively measurable, increasing and such that $\{Y^{i}\}_{i\in I}$ are pairwise i.i.d.\ with the common c.d.f.\ 
$$
y\mapsto F_t(y):=P\{Y_t^i\leq y\}.
$$
Moreover, for a fixed interaction constant\footnote{We could more generally consider processes $\alpha^{i}$ which are nonlinear functions of $Y^{i}$ and $\rho^{-i}$ and possibly a common noise, as in~\cite{Nutz.16}. However, the increased generality does not seem to lead to additional insights regarding the main questions of this paper, so we have chosen to use the simplified ``toy model'' in our exposition. The constant $c$ could in fact be normalized to $1$ by changing $Y^{i}$ and $r$, but we find it useful to represent the strength of interaction explicitly.} $c>0$,
$$
  \alpha^i(t)=Y_t^i+c\rho_{n}^{-i}(t), \quad \mbox{where}\quad \rho_n^{-i}(t)=\frac{\#\{j\neq i:\, \tau^j\leq t\}}{n}
$$
is the fraction of other players\footnote{Once again, we have decided to exclude player~$i$ in order to focus on the game-theoretic aspect of multiplicity. If player~$i$ considers her own action; i.e., uses $\rho$ instead of $\rho^{-i}$, non-uniqueness can occur without other agents' involvement simply because of the direct feedback on the state process.} (from the perspective of $i$) that have already stopped, according to $(\tau^{j})_{j\neq i}$. 
Specializing from the previous section, an  $n$-player  equilibrium boils down to the process $\rho_{n}(t)=\#\{j:\,\tau^j\leq t\}/n$ where $\tau^{j}$ are as in~\eqref{eq:optStopTime}. In particular, if $\rho_{n}$ is an equilibrium and $(t,\omega)$ is such that $\rho_n(t)(\omega)=k/n$, then as the stopping times satisfy $\tau^i=\inf\{t:\alpha^i(t)\geq r\}$, we must have\footnote{We will often abbreviate $\#\{i\in I:\, \dots\}$ to $\#\{\dots\}$ in what follows.}
\begin{equation}\label{eq:NplayerEquilibCond}
\#\{Y_t^i(\omega)+c\smsp \frac{k-1}{n}\geq r\}=k \quad\mbox{and}\quad
\#\{Y_t^i(\omega)+c\smsp \frac{k}{n}< r\}=n-k.
\end{equation}
This condition is also sufficient, in the sense made precise in Remark~\ref{rk:NplayerEquilibCondSuff}.

Next, we sketch the structure of all equilibria $\rho_n(t)=\#\{i:  \tau^i\leq t\}/n$ of this game by a recursive construction, starting with $K=\emptyset$.
\begin{enumerate}
\item[1.] Suppose that at a given stopping time $t_{0}$, a group $K\subsetneq I$ of agents has already stopped. Then every remaining agent $i\notin K$ examines her criterion 
$$
  \theta_{K}^{i} = \inf\{t:\, Y^{i}_{t} + c\smsp \frac{\#K}{n}\geq r\}.
$$
If $\theta_{K}^{i}\leq t_{0}$, then player $i$ must stop immediately. We add $i$ to the set $K$ and repeat Step 1 until no further players are forced to stop. (By the monotonicity in $\#K$, it does not matter in which order the agents are processed.)

\item[2.] Beyond individual players forced to stop, a group $J\subseteq K^{c}$ of agents may be able to ``coordinate'' and stop together.\footnote{While we are using suggestive language here, it should be noted that these are simply different configurations which may be equilibria. We are not trying to model a mechanism how players ``find'' an equilibrium.} Indeed, suppose that
$$
  \theta_{K}^{J} = \inf\{t:\, Y^{i}_{t} + c\smsp \frac{\#K+\#J-1}{n}\geq r\}
$$
satisfies $\theta_{K}^{J}\leq t_{0}$ for all $i\in J$. Then it is optimal for all these agents to stop as a group, and they may or may not ``choose'' to do so. If they stop, we add $J$ to $K$ and repeat the procedure starting with Step 1.

\item[3.] After all remaining groups of agents have decided whether to stop at time $t_{0}$, we increment time until there exists a group or individual agent wanting to stop, and start again at Step 1.
\end{enumerate}

The multiplicity of equilibria of this game arises because of the choices taken by the groups $J$ in Step 2, as well as the order in which the groups are processed. Next, we describe two of these equilibria in detail.
The first one is the minimal equilibrium and corresponds to groups~$J$ in Step 2 always choosing not to stop. This is equivalent to all players remaining in the game until their own optimality criterion forces them to quit.

\begin{proposition}
\label{pr:CharacterizationOfNMinEquilibrium}
There exists an $n$-player equilibrium $\rho^m_n$ such that
\begin{equation}\label{eq:CharNMinEquilib}
\rho^m_n(t)=\frac{k}{n}\Longleftrightarrow\left\{
\begin{aligned}
&\#\{Y_t^i+c\smsp \frac{k}{n}\geq r\}=k \\
&\#\{Y_t^i+c\smsp \frac{k-l}{n}\geq r\}\geq k-l+1,\quad 1\leq l\leq k.
\end{aligned}
\right.
\end{equation}
This equilibrium is minimal; i.e., $\rho^m_n(t)\leq\rho_n(t)$ for any $n$-player equilibrium~$\rho_n$.
\end{proposition}

\begin{proof}
  The construction is iterative. Given a set $K\subsetneq I$ corresponding to players who have already stopped, we can consider for all $i\notin K$ the stopping times 
  $$
    \theta_{K}^{i} = \inf\{t:\, Y^{i}_{t} + c\smsp \frac{\#K}{n}\geq r\}
  $$
  with the corresponding order statistics $\theta_{K}^{(1)}\leq\theta_{K}^{(2)} \leq \dots$. We define $\theta_{K}=\theta_{K}^{(1)}$ and $i_{K}=(1)$. 
  We note that agent $i$ must stop at $\theta_{K}^{i}$, even if no further agents $j\notin K$ choose to stop, and that $i_{K}$ is the first of the agents $i\notin K$ subject to this event.
  
  To define the equilibrium, start with $K_{0}=\emptyset$ and set $\tau^{i}=\theta_{K_{0}}\equiv \theta_{K_{0}}^{(1)}$ on $\{i=i_{K_{0}}\}$. Next, set $K_{1}=\{i_{K_{0}}\}$  and $\tau^{i}=\max\{\theta_{K_{1}},\theta_{K_{0}}\}$ on $\{i=i_{K_{1}}\}$, and continue inductively setting $K_{k}=K_{k-1}\cup \{i_{K_{k-1}}\}$ and $\tau^{i}=\max\{\theta_{K_{k}},\tau^{i_{K_{k-1}}}\}$ on $\{i=i_{K_{k}}\}$ for $k=2,\dots,n-1$. (The maximum needs to be taken since all the $\alpha^{j}$ are increased after player $i_{K_{k-1}}$ stops.)
  
  Setting $\rho^m_n(t)=\#\{i: \tau^i\leq t\}/n$, we have by construction that $\rho^m_n$ is an equilibrium with corresponding optimal stopping times $(\tau^{i})$ and that~\eqref{eq:CharNMinEquilib} holds.
  
To see the minimality, let $\rho_{n}$ be any $n$-player equilibrium and consider $(t,\omega)$ such that $\rho_n(t)(\omega)=k/n$. 
Let $k'$ be such that $\rho^m_n(t)(\omega)=k'/n$. If we had $k' > k$, then~\eqref{eq:CharNMinEquilib} would imply $\#\{Y_t^i(\omega)+c\smsp \frac{k}{n}\geq r\}\geq k+1$ and hence $\#\{Y_t^i(\omega)+c\smsp \frac{k}{n}<r\}\leq n-k-1$, a contradiction to~\eqref{eq:NplayerEquilibCond}. Thus, $k'\leq k$ and we have shown that $\rho^m_n\leq \rho_n$.
\end{proof}

\begin{remark}\label{rk:minimalFromGivenEquilibrium}
  Let $\rho$ be an $n$-player equilibrium and $t_{0}$ a stopping time. There exists an equilibrium which is minimal among all $n$-player equilibria~$\varrho$ such that $\varrho=\rho$ on $[0,t_{0}]$. Indeed, it is obtained by agents stopping as in $\rho$ until $t_{0}$, whereas from $t_{0}$ onwards we apply the construction in the proof of Proposition~\ref{pr:CharacterizationOfNMinEquilibrium} starting with $K=\{i:\,\tau^{i}\leq t_{0}\}$. We call this $\varrho$ the \emph{minimal extension} of $\rho$ after~$t_{0}$.
\end{remark}

The second extremal equilibrium is maximal and corresponds to players coordinating their actions such as to stop as early as possible. As seen in the construction below, this is equivalent to all players constantly seeking (maximally large) groups of collaborators so that immediate simultaneous stopping is optimal for all agents in the group.

\begin{proposition}
\label{pr:CharacterizationOfNMaxEquilibrium}
There exists an $n$-player equilibrium $\rho^M_n$ such that
\begin{equation}\label{eq:CharNMaxEquilib}
\rho^M_n(t)=\frac{k}{n}\Longleftrightarrow\left\{
\begin{aligned}
&\#\{Y_t^i+c\smsp \frac{k-1}{n}\geq r\}=k \\
&\#\{Y_t^i+c\smsp \frac{k+l-1}{n}\geq r\}\leq k+l-1,\quad 1\leq l\leq n-k.
\end{aligned}
\right.
\end{equation}
This equilibrium is maximal; i.e., $\rho^M_n(t)\geq\rho_n(t)$ for any $n$-player equilibrium~$\rho_n$.
\end{proposition}

\begin{proof}
  Given a set $K\subsetneq I$ of size $k=\#K$ corresponding to players who have already stopped, we can consider for $1\leq l \leq n-k$ the stopping times
  $$
    \theta_{K}^{l} = \inf\{t:\, \#\{i\notin K:\, Y_t^i+c\smsp \frac{k+l-1}{n}\geq r\}\geq l\};
  $$
  intuitively, this is the first time an additional group $J$ of $\#J=l$ agents can collectively stop.
  If $\theta_{K}^{(1)}\leq\dots\leq \theta_{K}^{(n-k)}$ are the corresponding order statistics (ties are split by assigning the lower rank to the larger index $l$), pick $l=(1)$ and let $J=J(K)$ be the set of $i\notin K$ such that $\{Y^{i}_{\theta_{K}^{l}}\}_{i\in J}$ are the $l$ largest elements in $\{Y^{i}_{\theta_{K}^{l}}\}_{i\in K^{c}}$; we think of $J$ as the $l$ most pessimistic agents remaining at time $\theta_{K}^{l}$ and denote $\theta_{K}:=\theta_{K}^{l}$.
    
  To define the equilibrium, start with $K_{0}=\emptyset$ and set $\tau^{i}=\theta_{\emptyset}$ for $i\in J(\emptyset)$. Next, set $K_{1}=J(\emptyset)$  and $\tau^{i}=\theta_{K_{1}}$ for $i\in J(K_{1})$, and continue inductively with $K_{2}=J(K_{1})\cup K_{1}$.
  Setting $\rho^M_n(t)=\#\{i: \tau^i\leq t\}/n$, we have by construction that $\rho^M_n$ is an equilibrium with corresponding optimal stopping times $(\tau^{i})$ and that~\eqref{eq:CharNMaxEquilib} holds.
  
To see the maximality, let $\rho_{n}$ be any $n$-player equilibrium and consider $(t,\omega)$ such that $\rho_n(t)(\omega)=k/n$. Again, $\rho_{n}$ must satisfy~\eqref{eq:NplayerEquilibCond}.
Let $k'$ be such that $\rho^M_n(t)(\omega)=k'/n$. If we had $k' < k$, then~\eqref{eq:CharNMaxEquilib} would imply that $\#\{Y_t^i(\omega)+c\smsp \frac{k-1}{n}\geq r\}\leq k-1$, contradicting~\eqref{eq:NplayerEquilibCond}.
\end{proof}

The following observations will be used in Section~\ref{se:convergenceGeneral} when we construct $n$-player equilibria converging to a given mean field equilibrium.

\begin{remark}\label{rk:cutAndPaste}
  (i) Consider $n$-player equilibria $\rho$ and $\rho'$, a stopping time $t_{0}$ and assume that $\rho(t_{0})\leq \rho'(t_{0})$. Then there exists an $n$-player equilibrium $\varrho$ such that 
  $$
    \varrho\1_{[0,t_{0})}=\rho \1_{[0,t_{0})} \quad \mbox{and}\quad \varrho\1_{[t_{0},\infty)}=\rho'\1_{[t_{0},\infty)}.
  $$
  Indeed, let $I_{0}$ be the set of agents that have stopped by time $t_{0}$ in equilibrium~$\rho$ and let $I_{1}$ be the analogue for $\rho'$. By~\eqref{eq:optStopTime} we necessarily have $I_{0}\subseteq I_{1}$. The equilibrium $\varrho$ is obtained by following the stopping times of $\rho$ on $[0,t_{0})$. At $t_{0}$, all agents in the group $J=I_{1}\setminus I_{0}$ stop (and this must be optimal as $\rho'$ is an equilibrium). After that the remaining agents act as in~$\rho'$.
  
  (ii) Extending the above, consider $n$-player equilibria $\rho$ and $\rho'$, stopping times $t_{0} \leq t_{1}$ and assume that $\rho(t_{0})\leq \rho'(t_{1})$. Then there exists an $n$-player equilibrium $\varrho$ such that 
  \begin{equation}\label{eq:cutPastProperty}
    \varrho\1_{[0,t_{0})}=\rho \1_{[0,t_{0})} \quad \mbox{and}\quad \varrho\1_{[t_{1},\infty)}=\rho'\1_{[t_{1},\infty)}.
  \end{equation}
  Indeed, let $\rho_{1}$ be the minimal extension of $\rho$ after $t_{0}$ (cf.\ Remark~\ref{rk:minimalFromGivenEquilibrium}). Let $I_{0}$ be the set of agents that have stopped by time $t_{0}$ in equilibrium $\rho$ and let $I_{1}$ be the set of agents that have stopped by time $t_{1}$ in equilibrium $\rho'$. Again, we observe that $I_{0}\subseteq I_{1}$, due to~\eqref{eq:optStopTime} and the increase of $Y^{i}$. Moreover, $I_{1}$ must include all agents that stop in the construction of the minimal extension on $[t_{0},t_{1}]$. As a result, $\rho_{1}(t_{1})\leq \rho'(t_{1})$, and now the claim follows by applying~(i).
  
  (iii) A last generalization is that when $\rho(t_{0})\leq \rho'(t_{1})$ merely holds on some set $A\in\cG_{t_{1}}$, then we can still construct an $n$-player equilibrium $\varrho$ satisfying~\eqref{eq:cutPastProperty} on~$A$. Indeed, $\varrho$ is found as in (ii) except that on $A^{c}$, agents continue to stop according to $\rho_{1}$ after $t_{1}$.
\end{remark}

\begin{remark}\label{rk:NplayerEquilibCondSuff}
  (i) The necessary condition~\eqref{eq:NplayerEquilibCond} is sufficient in the following sense. Fix $n$ and a stopping time $t_{0}$, and suppose there exists an $\cG_{t_{0}}$-measurable random variable $k$ satisfying~\eqref{eq:NplayerEquilibCond} at $t_{0}$; i.e.,
  \begin{equation*}
    \#\{Y_{t_{0}}^i+c\smsp \frac{k-1}{n}\geq r\}=k \quad\mbox{and}\quad
    \#\{Y_{t_{0}}^i+c\smsp \frac{k}{n}< r\}=n-k.
  \end{equation*}  
  Then there exists an $n$-player equilibrium $\varrho$ such that $\varrho(t_{0})=k/n$.
  
  To construct $\varrho$, let agents stop as in the minimal equilibrium $\rho^{m}_{n}$ up to time $t_{0}$. By the argument at the end of the proof of Proposition~\ref
{pr:CharacterizationOfNMinEquilibrium}, we must have $\rho^{m}_{n}(t_{0})\leq k/n$. At $t_{0}$, all remaining agents $i$ with $Y^{i}_{t_{0}} + c\smsp \frac{k-1}{n}\geq r$ stop, so that $\rho(t_{0})=k/n$. After that, the remaining agents follow the construction in the proof of Proposition~\ref{pr:CharacterizationOfNMinEquilibrium} starting with $K=\{i:\,\tau^{i}\leq t_{0}\}$.

 (ii) A variant of this holds when~\eqref{eq:NplayerEquilibCond} is satisfied on some set $A\in \cG_{t_{0}}$, with the conclusion that $\varrho(t_{0})=k/n$ holds only on $A$. Indeed, we construct $\varrho$ as above on $A$, whereas on $A^{c}$ we use $\rho^{m}_{n}$.

  (iii) For later use, we observe that if this construction is applied for two times $t_{0}\leq t_{1}$ and corresponding random variables $k_{0}\leq k_{1}$, the resulting equilibria satisfy $\varrho_{0}\leq\varrho_{1}$.
\end{remark}

\section{The Mean Field Game}
\label{se:MFG}

The game considered in this section is the ``toy model'' mean field game of \cite[Section~4]{Nutz.16}. Indeed, $(I, \cI, \lambda)$ is an atomless probability space and we
work on a so-called Fubini extension $(I\times\Omega,\Sigma,\mu)$ of the product $(I\times\Omega,\cI\times\cG,\lambda\times P)$; see \cite[Section~3]{Nutz.16}. For each $i\in I$, let $Y^i_t\geq 0$ be a right-continuous, increasing, $\G$-progressively measurable process such that for each $t\geq 0$, $(i,\omega)\mapsto Y_t^i(\omega)$ is $\Sigma$-measurable and $Y_t^i,\ i\in I$ are $\lambda$-essentially pairwise i.i.d.; see also \cite[Definition 3.1]{Nutz.16}. Working on a Fubini extension ensures that such processes exist, as well as the validity of an Exact Law of Large Numbers. In all that follows, we assume that the c.d.f.\ $y\mapsto F_t(y)=P\{Y_t^i\leq y\}$ is continuous.

Since $\lambda$ is atomless, each individual agent has zero mass and hence does not influence the state process $\rho(t)=\lambda\{i: \tau^i\leq t\}$. In particular, we do not distinguish $\rho$ and $\rho^{-i}$ and simply set $\alpha^i(t)=Y_t^i+c\rho(t)$. We recall that $\rho$ is an equilibrium if $\rho(t) = \lambda\{i:\, \tau^{i}\leq t\}$ where $\tau^{i}$ is as in~\eqref{eq:optStopTime} for $\lambda$-a.e.\ $i\in I$. Such a process may be random (see also~\cite{Nutz.16}). However, as common in the mean field game literature, we pay special attention to equilibria which are deterministic due to the infinite number of players.\footnote{Note that the key message of this paper, namely that some mean field equilibria are not limits of $n$-player equilibria, is only amplified if more mean field equilibria are considered.} The following is an improved version of \cite[Proposition 4.1]{Nutz.16} with necessary and sufficient conditions.

\begin{proposition}\label{pr:MFGequilib}
  A real function $\rho: \R_{+}\to[0,1]$ is a mean field game equilibrium if and only if it is increasing, right-continuous and 
\begin{equation}\label{eq:master}
  \rho(t)+F_t(r-c\rho(t))=1, \quad t\geq0.
\end{equation}
\end{proposition}

\begin{proof}
  Suppose that $\rho$ is a mean field game equilibrium, then $\rho$ is clearly increasing. Since $Y_t^i,\ i\in I$ are $\lambda$-essentially pairwise i.i.d., the Exact Law of Large Numbers (e.g., \cite[Section~3]{Nutz.16}) states that 
  $\lambda\{i:Y^i_t\leq u\}=F_{t}(u)$ for all $u$. Using also~\eqref{eq:optStopTime} and that $y\mapsto F_{t}(y)$ is continuous, we have
\begin{equation}\label{eq:proofRhoRC}
\rho(t)=\lambda\{i:\tau^i\leq t\}=\lambda\{i:Y^i_t+c\rho(t+)\geq r\}=1-F_t(r-c\rho(t+)).
\end{equation}
Recall that $Y^{i}$ has right-continuous paths. Using again the continuity of $F_{t}$, this implies that
\begin{equation}\label{eq:jointRC}
  \mbox{$(t,u)\mapsto F_t(r-cu)$ is jointly right-continuous.}
\end{equation}
It follows that $t\mapsto 1-F_t(r-c\rho(t+))$ is right-continuous, and thus the left-hand side of~\eqref {eq:proofRhoRC} must also be right-continuous. That is, $\rho(t)=\rho(t+)$, and then~\eqref{eq:proofRhoRC} becomes~\eqref{eq:master}.

Conversely, suppose that $\rho$ is a function with the stated properties. Defining the corresponding optimal stopping times $\tau^{i}$ as in~\eqref{eq:optStopTime}, the Exact Law of Large Number shows that
  $$
  \lambda\{i:\, \tau^{i}\leq t\}=\lambda\{i:\, Y^{i}_{t} + c\rho(t)\geq r\}
  =1-F_{t}(r-c\rho(t))=\rho(t);
  $$
  that is, $\rho$ is an equilibrium.
\end{proof}

The following notions will be crucial in determining the convergence to the mean field limit.

\begin{definition}\label{de:transversalDef}
  Fix $t\geq0$. A solution $u\in[0,1]$ of $u+F_t(r-cu)=1$ is called \emph{left-increasing-transversal} (or left-transversal for short)  if
  \begin{equation}\label{eq:leftTrans}
  \mbox{for all $\eps>0$ there is $u' \in(u-\eps,u)$ such that $u'+F_t(r-cu')<1$}
  \end{equation}
  and \emph{right-increasing-transversal} (or right-transversal) if
  \begin{equation}\label{eq:rightTrans}
  \mbox{for all $\eps>0$ there is $u' \in(u,u+\eps)$ such that $u'+F_t(r-cu')>1$.}
  \end{equation}
  It is called \emph{increasing-transversal} if both~\eqref{eq:leftTrans} and~\eqref{eq:rightTrans} hold, and \emph{decreasing-transversal} if these hold with the inequality signs reversed.
\end{definition}

For instance, in Figure~\ref{fig:Comparison of special solutions}, $u^{m}$ is left-increasing-transversal and $u^{mrt}, u^{M}$ are right-increasing-transversal, but only $u^{Mlt}$ is increasing-transversal. A decreasing-transversal solution is also depicted.
Next, we introduce a quartet of solutions that will be important in Section~\ref{se:convergenceExtremal}.

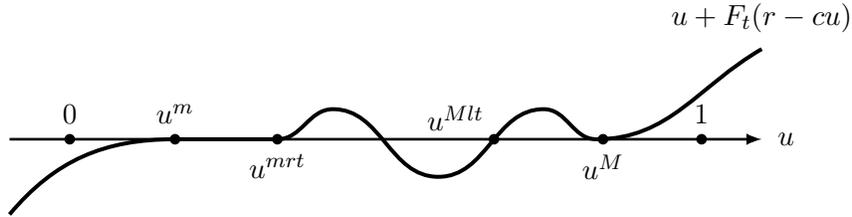
\begin{figure}[h]
\begin{center}
\begin{tikzpicture}
\draw[line width=1pt,-latex] (-5,0) to (5,0) node[right=2pt] {$u$};
\fill (-4.2,0) circle (2pt) node[above=2pt] at (-4.2,0) {0};
\fill (4.2,0) circle (2pt) node[above=2pt] at (4.2,0) {1};
\draw[line width=1.5pt] (-5,-1) to[out=50,in=180] (-2.8,0) to[out=0,in=180] (-1.44,0) to[out=0,in=180] (-0.7,.4) 
          to[out=0,in=180] (0.7,-.5) to[out=0,in=180] (2.1,.4) to[out=0,in=180] (2.8,0) to[out=0,in=210] (5,1.2) node[above=2pt] {$u+F_t(r-cu)$};
\fill (-2.8,0) circle (2pt) node[above=2pt]{$u^{m}$};
\fill (-1.44,0) circle (2pt) node[below=2pt]{$u^{mrt}$};
\fill (1.44,0) circle (2pt) node[above=8pt,left]{$u^{Mlt}$};
\fill (2.89,0) circle (2pt) node[below=2pt]{$u^{M}$};
\end{tikzpicture}
\end{center}
\caption{Solutions $u^m$, $u^{mrt}$, $u^{Mlt}$ and $u^M$}
\label{fig:Comparison of special solutions}
\end{figure}

\begin{lemma}\label{le:solutionsEx}
  Fix $t\geq0$. The equation $u+F_t(r-cu)=1$ has a minimal solution $u^{m}\in[0,1]$, a maximal solution $u^{M}\in[0,1]$, a minimal right-transversal solution $u^{mrt}\in[0,1]$, and a maximal left-transversal solution $u^{Mlt}\in[0,1]$.
\end{lemma}

\begin{proof}
  Since $G(u):=u+F_t(r-cu)<1$ for $u<0$ and $G(u)>1$ for $u>1$, the existence of $u^{m}$ and $u^{M}$ is immediate from the continuity of $G$. The fact that $G(u)<1$ for all $u<u^{m}$ entails that $u^{m}$ is left-transversal, and since it follows directly from the definition that the set of left-transversal solutions is stable under increasing limits, it follows that $u^{Mlt}$ exists. The argument for $u^{mrt}$ is similar.
\end{proof}

As illustrated in Figure~\ref{fig:Comparison of special solutions}, these four solutions may be distinct, and while $u^{m}$ is automatically left-transversal, it can happen that $u^{mrt}$ is not. Similarly for $u^{M}$ and $u^{Mlt}$. We can also note that $u^{mrt}\leq u^{Mlt}$ may fail, say if the graph is replaced by a flat stretch on $[u^{m},u^{M}]$. But in more generic cases, and in particular whenever $u^{m}$ and $u^{M}$ are not local extrema, the quartet describes at most two distinct solutions $u^m=u^{mrt} \leq u^{Mlt}=u^M$ and these are then increasing-transversal.

In view of Lemma~\ref{le:solutionsEx} we may define, given $t\geq0$,
\begin{equation}\label{eq:partMFGequlibDefs}
  \rho^{m}(t)=u^{m},\quad\rho^{M}(t)=u^{M},\quad\rho^{mrt}(t)=u^{mrt},\quad\rho^{Mlt}(t)=u^{Mlt}.
\end{equation}

Using the increase of $Y_{t}$ and~\eqref{eq:jointRC}, one can check that $\rho^{m},\rho^{M},\rho^{mrt},\rho^{Mlt}$ are increasing, $\rho^{M}$ and $\rho^{mrt}$ are right-continuous, and $\rho^{m}$ and $\rho^{Mlt}$ are left-continuous (but not continuous in general).

\begin{corollary}\label{co:minMaxEquilib}
  (i) If $\rho: \R_{+}\to[0,1]$ is any increasing function such that~\eqref{eq:master} holds, then $\rho(t+)$ is an equilibrium.
  
  (ii) The functions $t\mapsto \rho^m(t+)$ and $t\mapsto\rho^M(t)$ are the minimal and maximal equilibria of the mean field game; i.e., they are equilibria and any other equilibrium $\rho$ satisfies $\rho^m(t+)\leq \rho(t)\leq \rho^M(t)$ for all $t\geq0$.
\end{corollary}

\begin{proof}
  (i) If $\rho$ is any increasing function such that~\eqref{eq:master} holds, then the joint right-continuity in~\eqref{eq:jointRC} implies that $\rho(t+)+F_t(r-c\rho(t+))=1$ for all $t\geq0$. It now follows from Proposition~\ref{pr:MFGequilib} that $\rho(t+)$ is an equilibrium.

  (ii) Both $\rho^m(t+)$ and $\rho^M(t)$ are equilibria by~(i). If $\rho$ is any equilibrium, then it is necessarily right-continuous by Proposition~\ref{pr:MFGequilib} and thus $\rho^m\leq \rho\leq \rho^M$ implies $\rho^m(t+)\leq \rho(t)\leq \rho^M(t)$ for all $t\geq0$.
\end{proof}

\section{Convergence to Extremal Equilibria}
\label{se:convergenceExtremal}

The main goal of the last two sections is to understand which mean field equilibria are limits of $n$-player equilibria. In brief, we will see that mean field equilibria described by increasing-transversal solutions of~\eqref{eq:master} (on a sufficiently large sets of times $t$) are such limits, whereas other equilibria need not be proper limits of $n$-player equilibria; they merely occur as parts of mixtures which are limits.

In this section, we focus on the convergence to the minimal and maximal mean field equilibria; the less straightforward interior case is treated in the next section. As a first step, we relate limits of arbitrary $n$-player equilibria to mean field equilibria at a fixed time. We will see in Example~\ref{ex:twoType} that such limits  need not be deterministic mean field equilibria as defined in the preceding section, hence the following result relates limits to mixtures of equilibria. This is in line with the results of~\cite{CarmonaDelarueLacker.17,Lacker.14} stating that $n$-player equilibria converge to ``weak'' equilibria of the mean field game, while also illustrating that randomization can indeed occur in a quite natural example.

Given a closed set $A\subseteq\R$, we say that a sequence $(\xi_n)$ of random variables is \emph{asymptotically concentrated} on $A$ if $\lim_{n\to \infty}P(\xi_n\in A_{\eps})=1$ for all $\eps>0$, where $A_{\eps}=\{x\in\R:\, d(x,A)<\eps\}$ is the open $\eps$-neighborhood of $A$. When $(\xi_n)$ is uniformly bounded, as it will be the case below, this is equivalent to any weak cluster point of $(\xi_n)$ being concentrated on $A$. Moreover, for $t\geq0$, we denote the solutions of~\eqref{eq:master} by
$$
  \cU(t)=\{u\in[0,1]: u+F_{t}(r-cu)=1\}.
$$

\begin{proposition}
\label{pr:convergenceGeneralEquilib}
Fix $t\geq 0$ and let $(\rho_{n})_{n\geq1}$ be a sequence of $n$-player equilibria. Then $\rho_n(t)$ is asymptotically concentrated on $\cU(t)$.
\end{proposition}

\begin{proof} 
  We first show that for any interval $[u_{0},u_{1}]\subseteq [0,1]$ such that $u\mapsto u+F_t(r-cu)$ is strictly smaller than $1$ on $[u_{0},u_{1}]$,
\begin{equation}\label{eq:emptyIntSupport}
  P(u_{0} + \eps' \leq \rho_n(t)\leq u_{1}-\eps')\to 0 \quad\mbox{for all}\quad \eps'>0.
\end{equation}
Indeed, let $u_{0}< u_{1}$ be as above. By increasing the value of $u_{1}$ if necessary, we may assume without loss of generality that $u\mapsto u+F_t(r-cu)$ attains its maximum over $[u_{0},u_{1}]$ at $u_{1}$. Given $0<\eps<u_{1}-u_{0}$, we can then choose by continuity some $u \in (u_{1}-\eps,u_{1})$ such that
\begin{equation}\label{eq:leftGlobalMax}
  u'+F_t(r-cu')\leq u+F_t(r-cu)<1 \quad\mbox{for all}\quad u_{0}\leq u'\leq u.
\end{equation}
Furthermore, setting
$$
\eps_n(x)=\frac{\#\{Y_t^i+cx\geq r\}}{n}-(1-F_t(r-cx)),\quad x\in\R
$$
 and $\eps_n=\sup_{x\in \mathbb{R}}\{|\eps_n(x)|\}$, we have $\eps_n\to 0$ a.s.\ by the uniform convergence in the Glivenko--Cantelli theorem.
Let $X_i=\mathbf{1}_{\{Y_t^i+cu\geq r\}}$, then 
\begin{equation}\label{eq:epsNdef}
  \frac{X_1+\cdots+X_n}{n}=1-F_t(r-cu)+\eps_n(u).
\end{equation}
Denote by $[x]$ the largest integer $k\leq x$. For any $[u_{0} n]+1 \leq l \leq [u n]$, let $Z_i^l=\mathbf{1}_{\{Y_t^i+c\smsp \frac{l}{n}\geq r\}}$, then similarly 
$$
  \frac{Z_1^l+\cdots+Z_n^l}{n}=1-F_t(r-c\smsp \tfrac{l}{n})+\eps_n(\tfrac{l}{n}).
$$
On the event $\{Z_1^l+\cdots+Z_n^l =l\}$, we then have
$$1+\eps_n(\tfrac{l}{n})=\frac{l}{n}+F_t(r-c\smsp \tfrac{l}{n})\leq u+F_t(r-cu)$$
by~\eqref{eq:leftGlobalMax} and thus
$$
\frac{X_1+\cdots+X_n}{n}=1-F_t(r-cu)+\eps_n(u)\leq u-\eps_n(\tfrac{l}{n})+\eps_n(u)\leq u+2\eps_n.
$$
Combining this observation with~\eqref{eq:NplayerEquilibCond}, we have for all $[u_{0} n]+1 \leq l \leq [u n]$ that
\begin{align*}
\Big\{\rho_n(t)=\frac{l}{n}\Big\}&\subseteq \Big\{\#\big\{Y_t^i+c\smsp \tfrac{l}{n}\geq r\big\}=l\Big\}\\
                                                             &=\Big\{\frac{Z_1^l+\cdots+Z_n^l}{n}=\frac{l}{n}\Big\}\\
                                                             &\subseteq \Big\{\frac{X_1+\cdots+X_n}{n}\leq u+2\eps_n\Big\}.
\end{align*}
Hence, 
$$\Big\{ \frac{[u_{0} n]+1}{n} \leq \rho_n(t)\leq \frac{[u n]}{n}\Big\}\subseteq \Big\{\frac{X_1+\cdots+X_n}{n}\leq u+2\eps_n\Big\}$$
and thus
\begin{align*}
  P\Big(\frac{[u_{0} n]+1}{n} \leq \rho_n(t)\leq \frac{[u n]}{n}\Big)
  &\leq P\Big(\frac{X_1+\cdots+X_n}{n}\leq u+2\eps_n\Big)\\
  &= P\Big(u+F_t(r-cu) \geq 1-2\eps_n+\eps_n(u)\Big)\to0
\end{align*}
by~\eqref{eq:epsNdef} and~\eqref{eq:leftGlobalMax}.
Since $\eps>0$ was arbitrary, this shows~\eqref{eq:emptyIntSupport}. 

In a symmetric way, one can show the analogue of~\eqref{eq:emptyIntSupport} for intervals where $u\mapsto u+F_t(r-cu)$ is strictly larger than $1$. Since for any $\eps>0$ the complement of $\cU(t)_{\eps}$ consists of finitely many intervals of one of these two types, the claim follows.
\end{proof}

Next, we narrow down the asymptotic support for the minimal and maximal $n$-player equilibria $\rho^m_n$ and $\rho^M_n$. We will see in Section~\ref{se:negativeResultsExtremal} that the following result is optimal and the limiting support is not a singleton in general. We recall the notation introduced in~\eqref{eq:partMFGequlibDefs}.

\begin{lemma}
\label{le:convergenceMinMaxEquilib}
Fix $t\geq 0$.
\begin{enumerate}
\item
The minimal $n$-player equilibrium $\rho^m_n(t)$ is asymptotically concentrated on $[\rho^m(t),\rho^{mrt}(t)]\cap \cU(t)$.
\item
The maximal $n$-player equilibrium $\rho^M_n(t)$ is asymptotically concentrated on $[\rho^{Mlt}(t),\rho^M(t)]\cap \cU(t)$.
\end{enumerate}
\end{lemma}

\begin{proof}
  (i) In view of Proposition~\ref{pr:convergenceGeneralEquilib} and the definition of $\rho^{m}(t)$, it suffices to show that 
  \begin{equation}\label{eq:proofMinMaxConv}
    P(\rho^m_n(t)\geq \rho^{mrt}(t)+\eps')\to 0 \quad\mbox{for all}\quad \eps'>0.
  \end{equation} 
  Let $\eps>0$. As $\rho^{mrt}(t)$ is right-transversal we can find $u \in (\rho^{mrt}(t),\rho^{mrt}(t)+\eps)$ such that $1-F_t(r-cu)<u$. For $n$ large enough, we then have $\rho^{mrt}(t)< [u n]/n\leq u$. Let $X_i=\mathbf{1}_{\{Y_t^i+cu\geq r\}}$, then 
$$
\frac{X_1+\cdots+X_n}{n}\to EX_i=1-F_t(r-cu)\quad\mbox{a.s.}
$$
by the Law of Large Numbers. Hence,
$$
\frac{X_1+\cdots+X_n}{n}-\frac{[u n]}{n}\to 1-F_t(r-cu)-u<0\quad\mbox{a.s.}
$$
Using also~\eqref{eq:CharNMinEquilib}, we conclude that
\begin{align*}
P(\rho^m_n(t)\geq u)&\leq P\Big(\rho^m_n(t)\geq \frac{[u n]}{n}\Big)\\
                                       &\leq P\Big(\#\big\{Y_t^i+c\smsp \frac{[u n]}{n}\geq r\big\}\geq [u n]\Big)\\
                                       &\leq P\Big(\frac{\#\{Y_t^i+cu\geq r\}}{n}\geq \frac{[u n]}{n}\Big)\\
                                       &=P\Big(\frac{X_1+\cdots+X_n}{n}-\frac{[u n]}{n}\geq 0\Big)\to 0.
\end{align*}
As $\eps>0$ was arbitrary, the above implies~\eqref{eq:proofMinMaxConv}.

(ii) The arguments are similar to~(i) and therefore omitted.
\end{proof}

Next, we introduce an appropriate notion of convergence for dynamic equilibria as required for our main results. Note that given an increasing function, its right- and left-continuous limits (and all functions between these) differ only by the allocation of the function value at the (countably many) jumps.
The fact that mean field equilibria are right-continuous, cf.\ Proposition~\ref{pr:MFGequilib}, reflects the fact that agents stopping at time $t$ are counted as having left the game at time $t$, whereas left-continuity would correspond to counting them as leaving immediately after~$t$. Since this difference is not fundamental, it seems reasonable to consider limits ``up to taking right-continuous versions.'' This has been accomplished by notions of so-called Fatou convergence, e.g.~\cite{Kramkov.96, Zitkovic.02}, in other areas of stochastic analysis.

For increasing functions $\varphi_{n},\varphi$ on $\R_{+}$, we have that $(\liminf_{n}\varphi_{n})(t+)=(\limsup_{n}\varphi_{n})(t+) = \varphi(t+)$ holds for all $t\in\R_{+}$ if and only if $\lim \varphi_{n}(t)=\varphi(t)$ for all $t$ in a dense subset $D\subseteq\R_{+}$. This motivates the following.

\begin{definition}\label{de:Fatou}
A sequence $(\rho_{n})_{n\geq1}$ of $n$-player equilibria  \emph{Fatou converges in probability} to a mean field equilibrium $\rho$ if there exists a dense set $D\subseteq\R_{+}$ such that $\rho_{n}(t)\to\rho(t)$ in probability for all $t\in D$.
\end{definition}

We note that by a diagonalization procedure, Fatou convergence in probability implies Fatou convergence a.s.\ along a subsequence $(n_{k})$, where the a.s.\ convergence is defined by direct analogy to the above. In particular, it then follows that the right-continuous versions of $\liminf_{k}\rho_{n_{k}}$ and $\limsup_{k}\rho_{n_{k}}$ coincide with $\rho$ a.s.

With these notions in place, we can establish the convergence of extremal equilibria in the increasing-transversal case. (Note that the extremal equilibria cannot be decreasing-transversal; they are either increasing-transversal or tangential.)

\begin{theorem}\label{th:convUnderH}
  Suppose that for all $t$ in a dense subset $D\subseteq \R_{+}$, the minimal solution $u\in[0,1]$ of $u+F_t(r-cu)=1$ is increasing-transversal. 
  Then the minimal $n$-player equilibria $\rho^{m}_{n}$ Fatou converge in probability to the minimal mean field equilibrium as $n\to\infty$.
  
  The analogous assertion holds for the maximal equilibria $\rho^{M}_{n}$.
\end{theorem}

\begin{proof}
  By the hypothesis, $\rho^m(t)=\rho^{mrt}(t)$ for $t\in D$. Thus, Lemma~\ref{le:convergenceMinMaxEquilib} implies that $\lim \rho^{m}_{n}(t)=\rho^m(t)=\rho^{mrt}(t)$ in probability for $t\in D$. The analogue holds for $\rho^{M}_{n}$.
\end{proof}

Next, we discuss the transversality condition in more detail. In fact, if uniqueness holds for the mean field game, the condition is automatically satisfied and we conclude the following.

\begin{corollary}\label{co:uniqueness}
The following are equivalent:
\begin{enumerate}
\item the mean field game has a unique equilibrium $\rho$,
\item the equation $u+F_t(r-cu)=1$, $u\in [0,1]$ has a unique solution for a dense set of $t\in\R_{+}$.
\end{enumerate}
In that case, any sequence $(\rho_{n})_{n\geq1}$ of $n$-player equilibria Fatou converges in probability to $\rho$.
\end{corollary}

\begin{proof}
  If (i) holds, then $\rho^{m}(t+)=\rho^{M}(t)$ for all $t\geq0$ by Corollary~\ref{co:minMaxEquilib}, and~(ii) follows since  $\rho^{m}(t+)=\rho^{m}(t)$ except at the (countably many) jumps of $\rho^{m}$. The converse holds because equilibria are right-continuous; cf.\ Proposition~\ref{pr:MFGequilib}. Finally, if $u+F_t(r-cu)=1$ has a unique solution, this solution is necessarily increasing-transversal since $u+F_t(r-cu)<1$ for $u<0$ and $u+F_t(r-cu)>1$ for $u>1$.
\end{proof}

While we will see below that the transversality condition in Theorem~\ref{th:convUnderH} cannot be dropped, we can argue that this condition holds for a generic choice of signals $Y^{i}$. More generally, we discuss the following hypothesis (again, note that the extremal solutions can never be decreasing-transversal).

\begin{definition}\label{de:H}
  We say that Hypothesis~(H) holds if for all $t$ in a dense subset of $\R_{+}$, any solution of $u\in[0,1]$ of $u+F_t(r-cu)=1$ is increasing-transversal or decreasing-transversal.
\end{definition}

While this hypothesis does not hold for all choices of $Y^{i}$, the exceptional set is small in the sense that a ``typical'' $F_{t}$ will not have a local extremum of $u\mapsto u+F_t(r-cu)$ at a solution of $u+F_t(r-cu)=1$, so that the latter must be transversal. As $t$ varies over $\R_{+}$, the non-transversal case is somewhat more likely to occur, but typically at only finitely many $t$ so that the hypothesis still holds. There seems to be no obvious way to quantify this. However, we state the following result which confirms the general intuition and shows that Hypothesis~(H) is always valid after a small perturbation of~$Y^{i}$.

\begin{proposition}
\label{prop:generic}
  For every $\delta>0$ there exists $0\leq \eps\leq\delta$ such that after replacing $Y^{i}_{t}$ with $Y^{i}_t+\eps$, Hypothesis~(H) is satisfied.
\end{proposition}

\begin{proof}
 Let us first observe that for any real function $f(x)$, the set of local minimum values $S=\{f(x):\, x$ is a local minimum of $f\}$ is countable. Indeed, for every $s\in S$ there is an open interval $I_s$ with rational endpoints such that $s=\min\{f(x):x\in I_s\}$. If $s,t\in S$ and $I_s=I_t$, then $s=t$, showing that $I: S\to \mathbb{Q}\times\mathbb{Q}$ is injective.

For fixed $t\geq 0$, denote by $S(t)$ the set of all local minimum and maximum values of $u\mapsto u+F_t(r-cu)-1$, then $\cup_{t\in \mathbb{Q}}S(t)$ is again countable. Thus, we can find a sequence $a_{k}\downarrow 0$ with $a_k \notin \cup_{t\in \mathbb{Q}_{+}}S(t)$. Set $\eps_k=ca_k$. Then, passing from $Y_{t}$ to $Y_t^{\eps_k}=Y_t+\eps_k$, the function under consideration is
$$
  u\mapsto u+F_t^{\eps_k}(r-cu)=u+F_t(r-cu-\eps_k)=(u+a_k)+F_t(r-c(u+a_k))-a_k.
$$
By the construction of $a_k$, we know that 1 is not a local extremum value of this function. However, if a solution of $u+F_t^{\eps_k}(r-cu)=1$ failed to be transversal, then 1 would be the value at a local extremum. 
\end{proof}

\subsection{Counterexamples}
\label{se:negativeResultsExtremal}

In this section, we illustrate that the assertion of Theorem~\ref{th:convUnderH} may fail without the transversality condition, and more generally that the intervals in Lemma~\ref{le:convergenceMinMaxEquilib} cannot be improved. The examples presented here are essentially static, meaning that $Y^{i}_{t}$ does not depend on $t$. For purely technical reasons, namely to ensure the finiteness of the optimal stopping times~\eqref{eq:optStopTime} as assumed throughout, we introduce a time horizon $T\in(0,\infty)$ at which $Y^{i}_{t}$ jumps to a value larger than $r$, thus ensuring that all players stop.

In the first example, we allow for atoms in the distribution of $Y^{i}_{t}$ to obtain an analytically tractable example. We argue below that the atoms are not essential to the observed phenomenon.

\begin{example}
\label{ex:twoType}
Let $r=c=1$ and let $Y^i_{t}=Y^i_{0}$, $0\leq t< T$ be constant i.i.d.\ processes such that $\Law(Y^i_t)=\frac{1}{2}\delta_{\frac{1}{2}}+\frac{1}{2}\delta_2$ for all $0\leq t<T$, and set $Y^{i}_{t}=2$ for $t\geq T$. Then the law of the minimal $n$-player  equilibrium $\rho^m_n(t)$ converges to $\frac{1}{2}\delta_{\frac{1}{2}}+\frac{1}{2}\delta_{1}$ for all $0\leq t <T$.
\end{example}

\begin{proof}
Proposition~\ref{pr:CharacterizationOfNMinEquilibrium} yields two cases for every $\omega$. If strictly less than $n/2$ of the realizations $\{Y_0^i(\omega),\, i=1,\dots, n\}$ equal 2, all players~$i$ with $Y_0^i(\omega)=2$ stop at $t=0$ and those with $Y_0^i(\omega)=1/2$ never stop. Whereas if $n/2$ or more of the realizations equal 2, then all agents stop at $t=0$. It follows that the law of $\rho^m_n(t)\equiv\rho^m_n(0)$ converges to $\frac{1}{2}\delta_{\frac{1}{2}}+\frac{1}{2}\delta_{1}$ as $n\to\infty$.
\end{proof}

The limit law $\frac{1}{2}\delta_{\frac{1}{2}}+\frac{1}{2}\delta_{1}$ can be seen as a mixture of the deterministic mean field equilibria $\rho^{m}(t)\equiv \frac{1}{2}$ and $\rho^{mrt}(t)\equiv 1$. In fact, with an appropriate definition allowing for randomized equilibria, this mixture is itself an equilibrium. However, a remarkable conclusion is that there are no $n$-player equilibria converging to the minimal equilibrium $\rho^{m}$.

\begin{corollary}\label{co:twoTypeMinNotLimit}
 In the context of Example~\ref{ex:twoType}, $\rho^{m}(t)$ is not a weak accumulation point of $n$-player equilibria, for any $0\leq t<T$.
\end{corollary}

\begin{proof}
Suppose that there exists a subsequence $\rho_{k}=\rho_{n_{k}}$ of $n_{k}$-player equilibria such that $\rho_{k}(t)\to\rho(t)=1/2$ weakly. Then $\rho_{k}(t)\geq\rho_{k}^{m}(t)$  and $\Law(\rho_{k}^{m}(t))\to \frac{1}{2}\delta_{\frac{1}{2}}+\frac{1}{2}\delta_{1}$ yield a contradiction.
\end{proof}

It may be useful to contrast this with the fact that $\rho^{m}$ is a limit of \emph{approximate} Nash equilibria. To wit, if all players~$i$ with $Y_0^i(\omega)=2$ stop at $t=0$ whereas those with $Y_0^i(\omega)=1/2$ do not stop until $T$, we obtain an approximate Nash equilibrium converging to $\rho^{m}$ as $n\to\infty$.

The following example is a smooth version of Example~\ref{ex:twoType} where $Y^{i}_{t}$ admits a density; see also Figure~\ref{fig:cdf of different density functions}(b). It is not analytically tractable but the qualitative behavior is  the same.

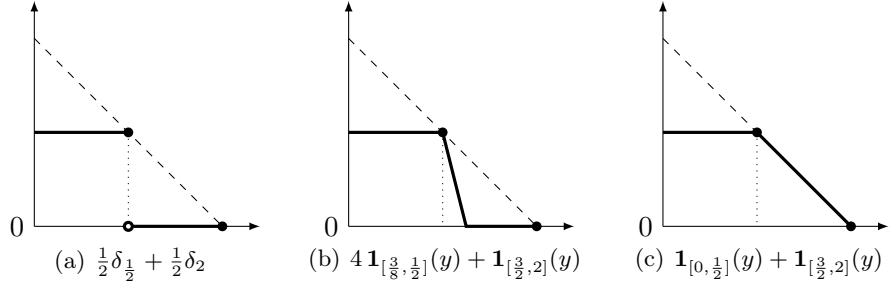
\begin{figure}
  \centering
\subfigure[$\frac{1}{2}\delta_{\frac{1}{2}}+\frac{1}{2}\delta_2$]{\label{fig:fft:a}
\begin{minipage}[c]{0.3\textwidth}
\centering
\begin{tikzpicture}
\draw[-latex] (0,0) -- (3,0);
\draw[-latex] (0,0) -- (0,3);
\node[below,left] at (0,0) {0};
\draw[dashed] (0,2.5) -- (2.5,0);
\draw[very thick] (0,1.25) -- (1.25,1.25) node[minimum width=2.5pt,inner sep=0,draw, fill,circle]{};
\draw[dotted] (1.25,1.25) -- (1.25,0);
\draw[very thick] (1.25,0) node[minimum width=3pt,inner sep=0,draw, fill=white,circle]{} -- (2.5,0) node[minimum width=2.5pt,inner sep=0,draw, fill,circle]{};
\end{tikzpicture}
\end{minipage}
}
\subfigure[$4\smsp \mathbf{1}_{[\frac{3}{8},\frac{1}{2}]}(y)+\mathbf{1}_{[\frac{3}{2},2]}(y)$]{\label{fig:fft:b}
\begin{minipage}[c]{0.3\textwidth}
\centering
\begin{tikzpicture}
\draw[-latex] (0,0) -- (3,0);
\draw[-latex] (0,0) -- (0,3);
\node[below,left] at (0,0) {0};
\draw[dashed] (0,2.5) -- (2.5,0);
\draw[dotted] (1.25,1.25) -- (1.25,0);
\draw[very thick] (0,1.25) -- (1.25,1.25) node[minimum width=2.5pt,inner sep=0,draw, fill,circle]{};
\draw[very thick] (1.25,1.25) -- (1.5625,0) -- (2.5,0) node[minimum width=2.5pt,inner sep=0,draw, fill,circle]{};
\end{tikzpicture}
\end{minipage}
}
\subfigure[$\mathbf{1}_{[0,\frac{1}{2}]}(y)+\mathbf{1}_{[\frac{3}{2},2]}(y)$]{\label{fig:fft:c}
\begin{minipage}[c]{0.3\textwidth}
\centering
\begin{tikzpicture}
\draw[-latex] (0,0) -- (3,0);
\draw[-latex] (0,0) -- (0,3);
\node[below,left] at (0,0) {0};
\draw[dashed] (0,2.5) -- (2.5,0);
\draw[dotted] (1.25,1.25) -- (1.25,0);
\draw[very thick] (0,1.25) -- (1.25,1.25) node[minimum width=2.5pt,inner sep=0,draw, fill,circle]{};
\draw[very thick] (1.25,1.25) -- (2.5,0) node[minimum width=2.5pt,inner sep=0,draw, fill,circle]{};
\end{tikzpicture}
\end{minipage}
}
\caption{Graphs of $F_t(1-u)$ (solid) and $1-u$ (dashed)}
\label{fig:cdf of different density functions}
\end{figure}

\begin{example}
\label{ex:twoTypeDensity}
Let $r=c=1$ and let $Y^i_{t}=Y^{i}_{0}$, $0\leq t< T$  be i.i.d.\ processes such that the law of $Y^i_t$ has the density $f_{t}(y)=4\smsp \mathbf{1}_{[\frac{3}{8},\frac{1}{2}]}(y)+\mathbf{1}_{[\frac{3}{2},2]}(y)$ for all $0\leq t<T$, and let $Y^i_{t}=2+X^{i},$ $t\geq T$, where $X^{i}$ are i.i.d.\ with a continuous distribution on $[0,1]$. Then the simulation of $\rho^{m}_{n}(t)$, cf.\ Figure~\ref{fig:fft:I}, shows that $\rho^m_n(t)$ again converges to $\frac{1}{2}\delta_{\frac{1}{2}}+\frac{1}{2}\delta_{1}$ for $0\leq t<T$ which is again a mixture of the deterministic mean field equilibria $\rho^{m}(t)\equiv \frac{1}{2}$ and $\rho^{mrt}(t)\equiv 1$.
\end{example}

In the third example, the mean field game admits a continuum of solutions; see also Figure~\ref{fig:cdf of different density functions}(c).

\begin{example}
\label{ex:twoTypeDensityInterpol}
Consider the setting of Example~\ref{ex:twoTypeDensity} with density $f_{t}(y)=\mathbf{1}_{[0,\frac{1}{2}]}(y)+\mathbf{1}_{[\frac{3}{2},2]}(y)$. In this case, we again have $\rho^{m}(t)\equiv \frac{1}{2}$ and $\rho^{mrt}(t)\equiv 1$, but now all values in between also correspond to mean field equilibria. The simulation of $\rho^{m}_{n}(t)$, cf.~Figure \ref{fig:fft:II}, illustrates that the law of $\rho^m_n(t)$ converges to a mixture of all these equilibria.
\end{example}

\begin{figure}[thb]
  \centering
\subfigure[density $4\smsp \mathbf{1}_{[\frac{3}{8},\frac{1}{2}]}(y)+\mathbf{1}_{[\frac{3}{2},2]}(y)$]{\label{fig:fft:I}
\begin{minipage}[c]{0.4\textwidth}
\centering
\includegraphics[width=2.2in,height=1.7in]{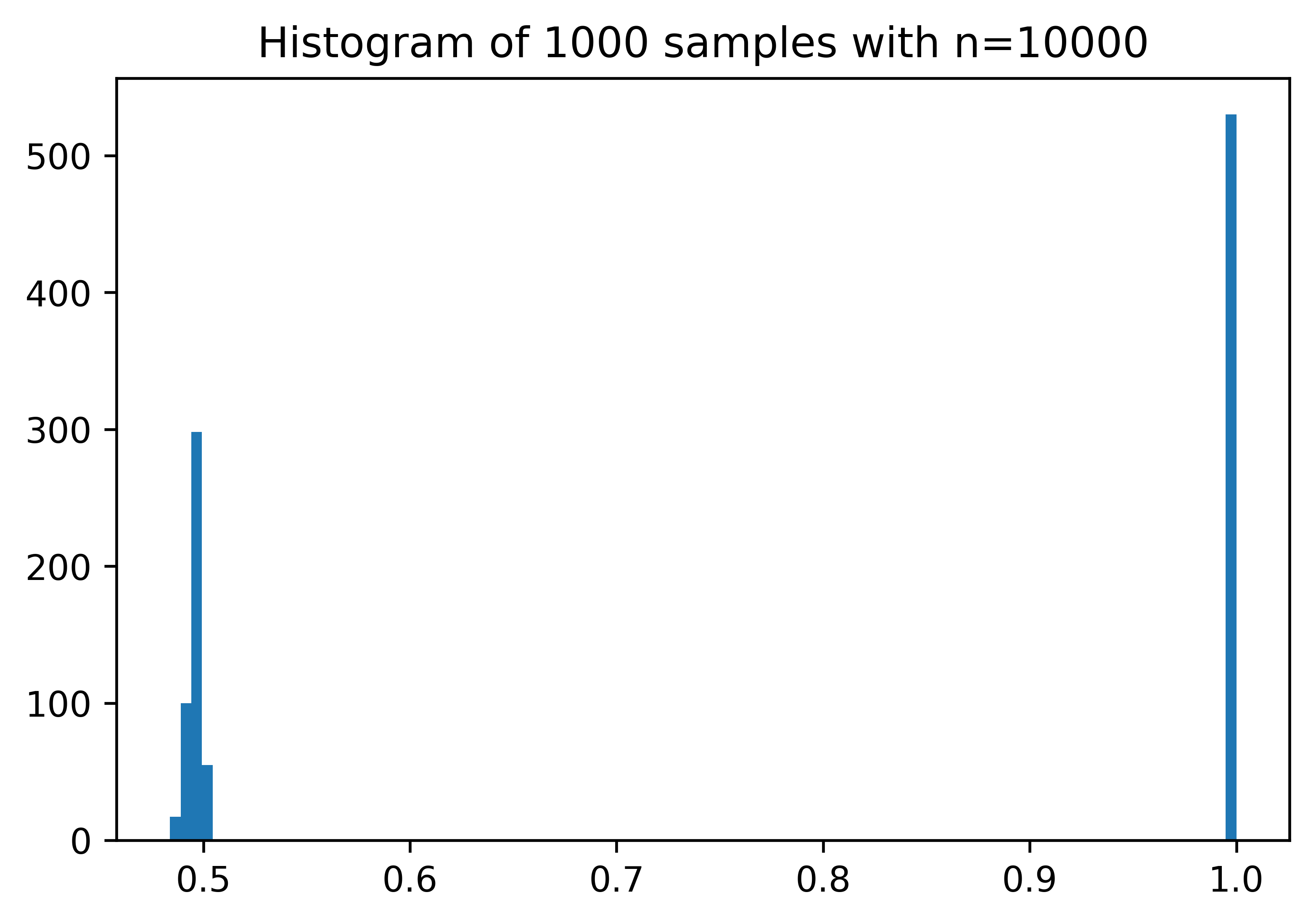}
\end{minipage}
}
\subfigure[density $\mathbf{1}_{[0,\frac{1}{2}]}(y)+\mathbf{1}_{[\frac{3}{2},2]}(y)$]{\label{fig:fft:II}
\begin{minipage}[c]{0.5\textwidth}
\centering
\includegraphics[width=2.2in,height=1.7in]{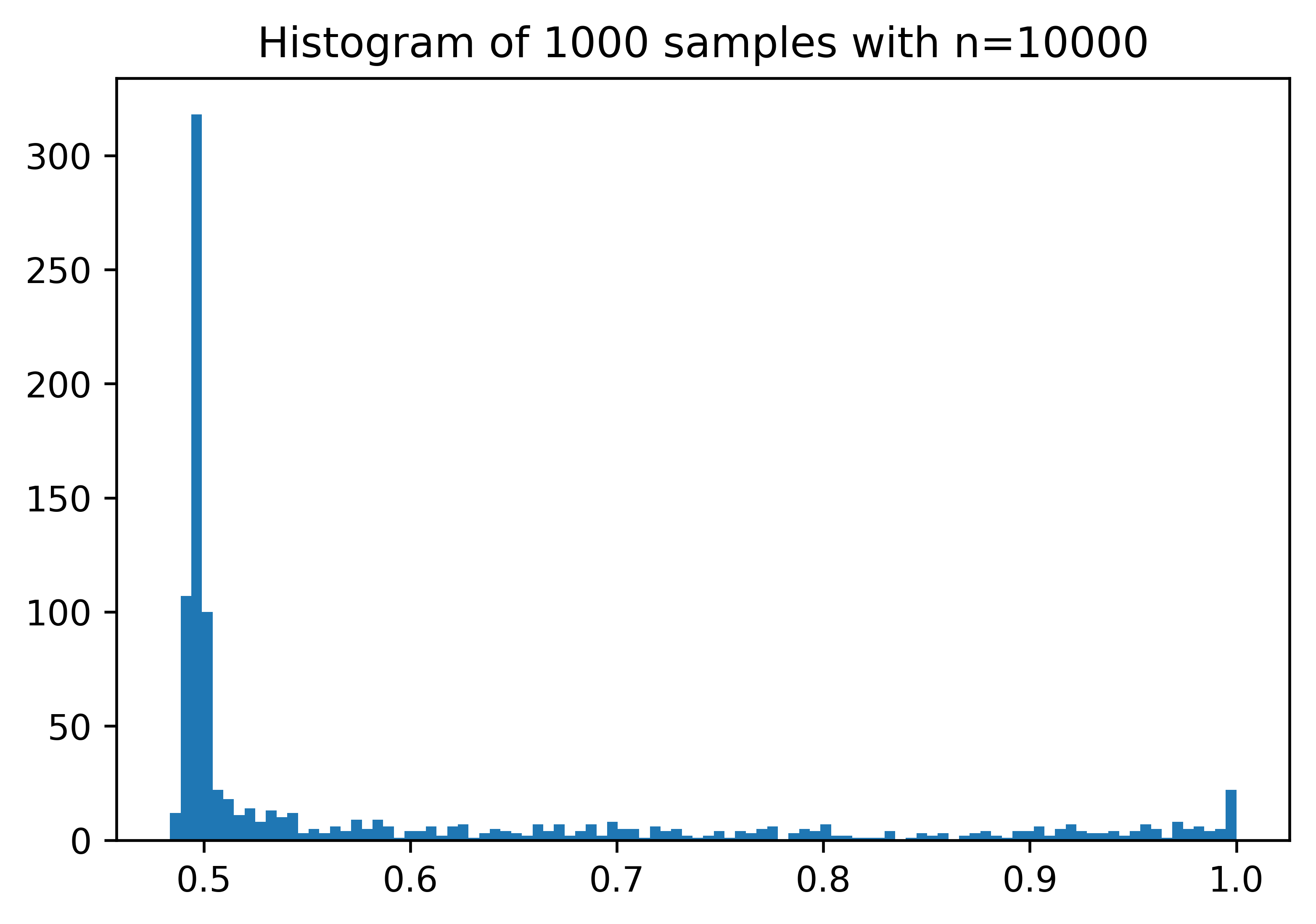}
\end{minipage}
}
\caption{Simulations for $n$-player minimal equilibria ($n=10'000$). Locations $k/n$ of equilibria with $k$ stopped players on the $x$-axis, number of samples with that equilibrium on the $y$-axis.}
\label{fig:simulations for $n$-player game minimal equilibrium}
\end{figure}

When the minimal mean field equilibrium is not increasing-transversal, the preceding examples illustrate that it need not be the limit of the minimal $n$-player equilibria. The final example shows that both cases are possible: it may be the limit even if it is not increasing-transversal.

\begin{example}
\label{ex:twoTypeDensityGood}
Consider the setting of Example~\ref{ex:twoTypeDensity} with density $f_{t}(y)=2\,\mathbf{1}_{[1/2,1]}(y)$. In this case, we easily compute that $\rho^{m}(t)\equiv 0$ and $\rho^{mrt}(t)\equiv 1$. Nevertheless, $\rho^{m}_{n}(t)\equiv0$ due to $Y^{i}_{t}<r$ a.s., and thus $\rho^{m}_{n}(t)\to\rho^{m}(t)$.
\end{example}

\section{Convergence to General Equilibria}\label{se:convergenceGeneral}

Theorem~\ref{th:convUnderH} shows that if the minimal and maximal mean field equilibria are increasing-transversal (on a dense set), then they are the limits of the minimal and maximal $n$-player equilibria. Indeed, the latter are obvious candidates for sequences converging to these mean field equilibria. For mean field equilibria that are not extremal, there are no obvious candidates for the approximating $n$-player equilibria. The following result shows that increasing-transversal equilibria are still limits; however, the approximating $n$-player equilibria have no simple description. We will see in Section~\ref{se:decrasingTransversal} that the analogue for decreasing-transversal solutions fails. 

\subsection{Increasing-Transversal Equilibria}

\begin{theorem}\label{th:convToIncreasing}
  Let $\rho$ be a mean field equilibrium. Suppose that for all $t$ in a dense subset $D\subseteq \R_{+}$, the solution $u:=\rho(t)$ of $u+F_t(r-cu)=1$ is increasing-transversal. Then there exist $n$-player equilibria $(\rho_{n})_{n\geq1}$ which Fatou converge in probability to $\rho$ as $n\to\infty$.
\end{theorem}

The first step of the proof is to solve a static version of the problem. This will be accomplished by a fixed point argument for monotone functions.

\begin{lemma}\label{le:convToIncreasingStatic}
  Let $t\geq0$, let $u\in[0,1]$ be an increasing-transversal solution of $u+F_t(r-cu)=1$ and let $\eps,\delta>0$. There are $n_{0}\in\N$ and $A\in\cG_{t}$ with $P(A)>1-\eps$ such that for all $n\geq n_{0}$ and $\omega\in A$, there exists $k(\omega)\in\N$ such that $|u-k(\omega)/n|\leq \delta$ and~\eqref{eq:NplayerEquilibCond} holds; i.e.,
  \begin{equation*}
    \#\{Y_t^i(\omega)+c\smsp \frac{k(\omega)-1}{n}\geq r\}=k(\omega) \quad\!\mbox{and}\!\quad
    \#\{Y_t^i(\omega)+c\smsp \frac{k(\omega)}{n}< r\}=n-k(\omega).
  \end{equation*}
  Moreover, $k(\omega)$ can be chosen as a measurable function of $Y^{1}_{t}(\omega),\dots, Y^{n}_{t}(\omega)$.
\end{lemma}

\begin{proof}
  Since $u$ is increasing-transversal, there are points $u_{0},u_{1}\in\R$ such that $u-\delta/2 \leq u_{0} < u < u_{1} \leq u+ \delta/2$ and 
  $$
    u_{0}< 1 - F_{t}(r-cu_{0}) \leq 1 - F_{t}(r-cu_{1}) < u_{1},
  $$
  where the inequality in the middle is due to the monotonicity of $F_{t}$.
  The Glivenko--Cantelli theorem then implies that the event $A_{n}$ consisting of all~$\omega$ such that
  $$
    [nu_{0}]\leq \#\{Y_t^i(\omega)+c\smsp \frac{[nu_{0}]-1}{n}\geq r\} \leq \#\{Y_t^i(\omega)+c\smsp \frac{[nu_{1}]}{n}\geq r\} \leq [nu_{1}]
  $$
  satisfies $P(A_{n})\to 1$. For fixed $n$ and $\omega\in A_{n}$, consider the integer-valued function
  $$
    k\mapsto G(k):=\#\{Y_t^i(\omega)+c\smsp \tfrac{k}{n}\geq r\}.
  $$
  By the above, $G$ maps $\{[nu_{0}]-1,[nu_{0}],\dots,[nu_{1}]\}$ into $\{[nu_{0}],\dots,[nu_{1}]\}$. Moreover, $G$ is monotone increasing. Lemma~\ref{le:fixedpoint} below then yields the existence of $[nu_{0}]\leq k \leq [nu_{1}]$ such that $G(k-1)=G(k)=k$ which is exactly~\eqref{eq:NplayerEquilibCond}. By the choice of $u_{0},u_{1}$ we also have $|u-k/n|\leq \delta$ for $n$ large. Moreover, it is clear from the proof of Lemma~\ref{le:fixedpoint} that $k$ is a measurable function of $Y^{1}_{t},\dots, Y^{n}_{t}$.
\end{proof}

\begin{lemma}\label{le:fixedpoint}
  Let $x_{0}<x_{1}<\dots < x_{N}$ be real numbers for some $N\geq 1$. Let $J=\{x_{1},\dots,x_{N}\}$ and $J_{0}=\{x_{0}\}\cup J$. If $f: J_{0}\to J$ is monotone increasing, there exists $k\in\{1,\dots,N\}$ such that $f(x_{k-1})=f(x_{k})=x_{k}$.
\end{lemma}

\begin{proof}
  Since $f$ is monotone and maps $J$ into $J$, it must have a fixed point in~$J$. We claim that the minimal $k\in\{1,\dots,N\}$ such that $f(x_{k})=x_{k}$ has the desired property. Indeed, if $k=1$, monotonicity implies that $f(x_{0})=f(x_{1})$ and the proof is complete. If $k>1$, we observe that $f(x_{l-1})\geq x_{l}$ for all $1\leq l\leq k$. Indeed, $f(x_{1})\geq x_{2}$ since $x_{1}$ is not a fixed point, but then $f(x_{2})\geq x_{3}$ since $x_{2}$ is not a fixed point and $f$ is monotone, and so on. In particular, $f(x_{k-1})\geq x_{k}$ and thus $f(x_{k-1})=f(x_{k})=x_{k}$.
\end{proof}

\begin{proof}[Proof of Theorem~\ref{th:convToIncreasing}]
  Fix $N\in\N$ and let $t_{1}<\dots<t_{N}$ be in $D$. For $n$ large enough, Lemma~\ref{le:convToIncreasingStatic} allows us to find sets $A_{l}\in\cG_{t_{l}}$ with $P(A_{l})>1-N^{-2}$ and random variables $k_{l}$ satisfying $|\rho(t_{l})-k_{l}/n|\leq \delta:=1/N$ and~\eqref{eq:NplayerEquilibCond} on~$A_{l}$, for $1\leq l \leq N$.
  
  Following Remark~\ref{rk:NplayerEquilibCondSuff}, we can construct $n$-player equilibria $\rho^{l}_{n}$ such that $\rho^{l}_{n}(t_{l})=k_{l}/n$ on $A_{l}$. Next, we argue that these $\rho^{l}_{n}$ can be chosen such that 
  \begin{equation}\label{eq:FPordered}
    \rho^{1}_{n}(t_{1})\leq \cdots  \leq \rho^{m}_{n}(t_{m}) \mbox{ on } A_{1}\cap\cdots\cap A_{m},\quad 1\leq m\leq N.
  \end{equation}  
  Indeed, we have $\rho(t_{l})\leq\rho(t_{l+1})$ by the increase of $\rho$. If $\rho(t_{l})<\rho(t_{l+1})$, then we can ensure $\rho^{l}_{n}(t_{l})\leq \rho^{l+1}_{n}(t_{l+1})$ on $A_{l}\cap A_{l+1}$ simply by choosing $\delta<|\rho(t_{l})-\rho(t_{l+1})|/2$ in Lemma~\ref{le:convToIncreasingStatic}. If $\rho(t_{l})=\rho(t_{l+1})$, we can observe that if the construction in the proof of Lemma~\ref{le:convToIncreasingStatic} is executed twice with $t_{l}$ and $t_{l+1}$, then by choosing the same parameters $u_{0},u_{1}$ the corresponding functions $f_{l}$ and $f_{l+1}$ satisfy $f_{l}\leq f_{l+1}$ due to the increase of $Y^{i}$. This implies that the corresponding minimal fixed points produced by the proof of Lemma~\ref{le:fixedpoint} satisfy $\rho^{l}_{n}(t_{l})\leq \rho^{l+1}_{n}(t_{l+1})$.
  
  In view of~\eqref{eq:FPordered}, we can use Remark~\ref{rk:cutAndPaste}(iii) to construct from the equilibria $(\rho^{l}_{n})_{1\leq l\leq N}$ another $n$-player equilibrium $\varrho_{n}$ with the property that $\varrho_{n}(t_{l})= \rho^{l}_{n}(t_{l})$ for all $1\leq l \leq N$ on $A^{N}:=\cap_{l=1}^{N}A_{l}$.
  
  To summarize, $\varrho_{n}$ satisfies $|\rho(t_{l})-\varrho_{n}(t_{l})|\leq 1/N$ for all $1\leq l \leq N$ on the set $A^{N}$ which has probability $P(A^{N})\geq 1-N^{-1}$. By letting $t_{1},\dots,t_{N}$ exhaust a countable dense subset $D'\subseteq D\subseteq \R_{+}$ as $N\to\infty$, this shows that there exist $n$-player equilibria  $(\varrho_{n})_{n\geq1}$ such that $\varrho_{n}(t)\to\rho(t)$ in probability for all $t\in D'$ and the proof is complete.
\end{proof}

\begin{remark}\label{rk:convergenceToMixtures}
  The construction leading to Theorem~\ref{th:convToIncreasing} is pathwise and thus extends beyond deterministic mean field equilibria. For instance, let $\rho^{1},\rho^{2}$ be such equilibria satisfying the assumption of Theorem~\ref{th:convToIncreasing}, let $\lambda\in[0,1]$ and suppose that the $n$-player game admits a set $A\in\cG_{0}$ with $P(A)=\lambda$. Then we can apply the construction separately on $A$ and $A^{c}$ to find $n$-player equilibria $\rho_{n}$ converging to the mixture $\lambda \delta_{\rho^{1}} + (1-\lambda)\delta_{\rho^{2}}$ on a dense set. In the same vain, convergence to more general mixtures could be analyzed.
\end{remark}

\subsection{Decreasing-Transversal Equilibria}
\label{se:decrasingTransversal}

Let us begin with a simulation and then establish that the observations correspond to a general result.

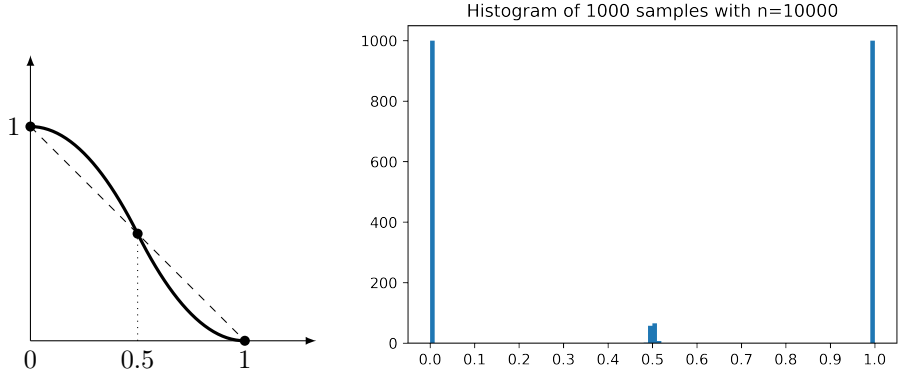
\begin{figure}[bh]
  \centering
\begin{tikzpicture}[scale=0.95]
\draw[-latex] (0,0) -- (4,0);
\draw[-latex] (0,0) -- (0,4);
\node[below] at (0,0) {\small 0};
\node[below] at (3,0) {\small 1};
\node[left] at (0,3) {\small 1};
\node[below] at (1.5,0) {\small $0.5$};
\draw[dashed] (0,3) -- (3,0);
\draw[very thick,domain=0:1.5] plot(\x,{3*(1-2*((\x)/3)^2)]});
\draw[very thick,domain=1.5:3] plot(\x,{6*(1-(\x)/3)^2});
\draw[dotted] (1.5,0) -- (1.5,1.5);
\fill (0,3) circle (2pt);
\fill (3,0) circle (2pt);
\fill (1.5,1.5) circle (2pt);
\end{tikzpicture}
\hspace{.5em}
\includegraphics[scale=.55]{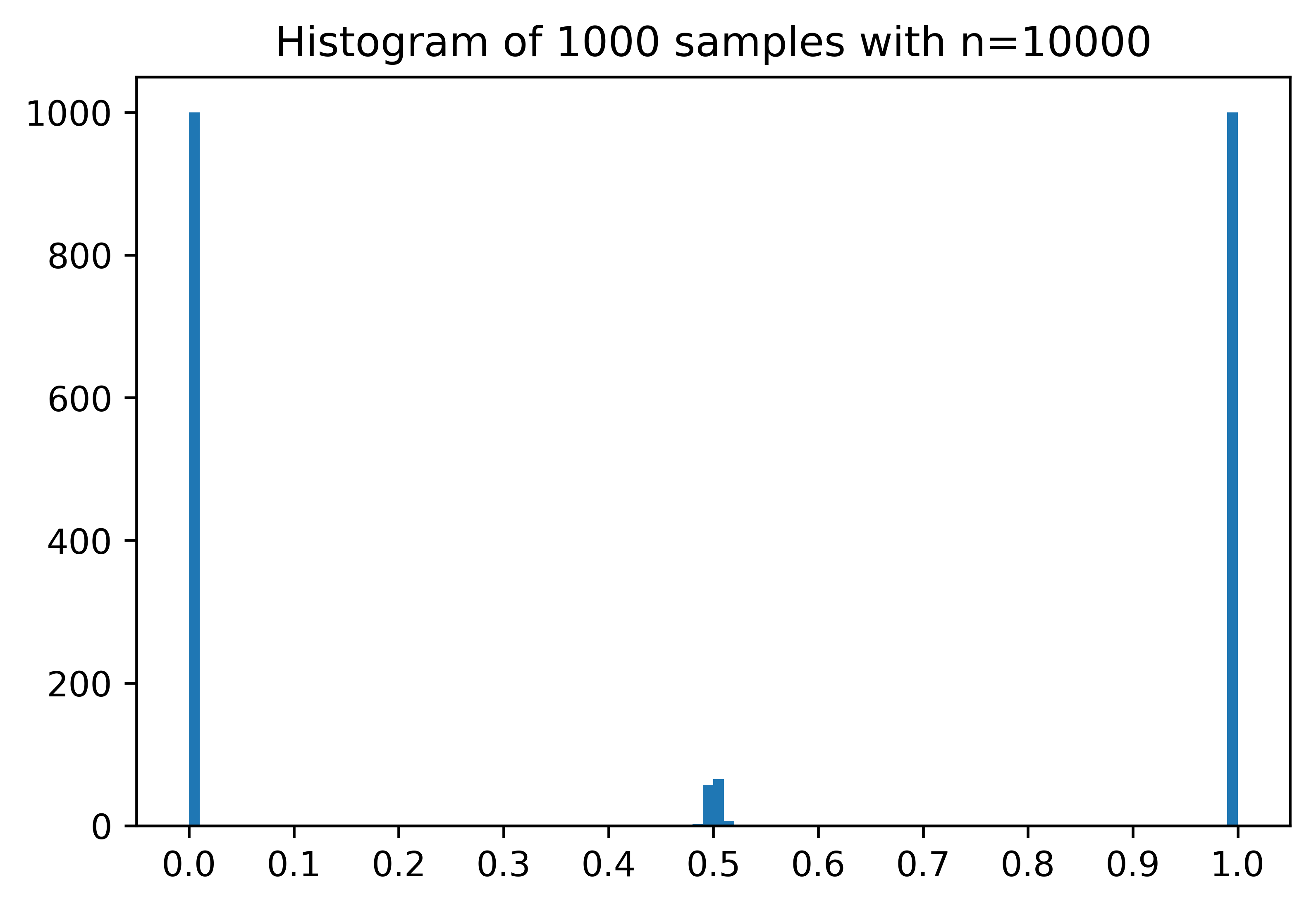}
\caption{C.d.f.\ and simulation of Example~\ref{ex:tent}. The decreasing-transversal equilibrium at $0.5$ can only be approximated on  12.5\% of the samples.}
\label{fig:tent}
\end{figure}

\begin{example}
\label{ex:tent}
Let $r=c=1$ and let $Y^i_{t}=Y^i_{0}$, $0\leq t< T$ be constant i.i.d.\ processes such that $\Law(Y^i_t)$ has the tent-shaped probability density $f(x)=2-4|x-1/2|$, $x\in[0,1]$. As illustrated in Figure~\ref{fig:tent} (left panel), the corresponding equation~\eqref{eq:master} has a decreasing-transversal solution at $u=1/2$ and increasing-transversal solutions at $u=0$ and $u=1$. For the game with $n=10'000$ players,  the histogram in Figure~\ref{fig:tent} shows the values of $k/n$ such that $k$ satisfies the equilibrium conditions~\eqref{eq:NplayerEquilibCond}. The simulation illustrates the convergence to the equilibria at $u=0,1$ as proved in Theorem~\ref{th:convToIncreasing} but also suggests that $u=1/2$ is not a limit of $n$-player equilibria; indeed, only about  12.5\% of the samples
allow for an $n$-player equilibrium with $k/n$ close to~$1/2$. In Proposition~\ref{pr:expectedNumber}, we will establish an asymptotic  upper bound which yields $e^{-2}\approx 13.5\%$ in this example.
\end{example}


In the remainder of this section we assume that $F_{t}$ admits a continuous density~$f_{t}$. Let $x\in[0,1]$ be a solution of $u+F_{t}(r-cu)=1$. We say that $x$ is \emph{strongly decreasing-transversal} if $\partial_{u}|_{u=x}[u+F_{t}(r-cu)]<0$ or equivalently
  $$
    f_{t}(r-cx)>c^{-1}.
  $$
  We note that $x$ is then necessarily in $(0,1)$ and decreasing-transversal in the sense of Definition~\ref{de:transversalDef}; the only difference (given the continuity assumption) is that we exclude the case where $u+F_{t}(r-cu)$ has a vanishing derivative at $x$ (see also Remark~\ref{rk:tangentialDecreasingCase}). Intuitively, when $f_{t}(r-cx)$ is large, there are many similar agents (in terms of values of $Y^{i}$ and relative to the interaction constant $c$) close to such a state. As a result, these agents may tend to coordinate and either all stop or all not stop: it may be impossible to break up the group\footnote{Clearly, this intuition does not explain the phase-transition character of the phenomenon. To gather the intuition for a large density, it may be useful to consider the limiting case of an atom in~$F_{t}$: all agents corresponding to the atom make the same stopping decision.} and create an $n$-player equilibrium close to $x$. 

\begin{theorem}\label{th:convToDecreasing}
Let $\rho$ be a mean field equilibrium and suppose that the set
$$
\{t\geq0:\, \rho(t) \mbox{ is strongly decreasing-transversal}\}
$$
has nonempty interior.\footnote{Note that the condition is nonempty interior rather than the set being nonempty. This corresponds to the fact that convergence in probability on a dense set of times~$t$ is sufficient for Fatou convergence; cf.\ Definition~\ref{de:Fatou}.} Then there does not exist a sequence of $n$-player equilibria $\rho_{n}$ Fatou converging to $\rho$ in probability.
\end{theorem}

This theorem follows from Corollary~\ref{co:decreasingLessOne} below which shows non-existence with positive probability at any fixed time $t$ where $\rho(t)$ is strongly decreasing-transversal. For brevity, we set
$$
  G_{n,t}(k)=\#\{Y_t^i+c\smsp \tfrac{k}{n}\geq r\}
$$
so that the $n$-player equilibrium conditions~\eqref{eq:NplayerEquilibCond} can be expressed concisely as
$G_{n,t}(k)=k=G_{n,t}(k-1)$.
Moreover, we introduce
\begin{align*}
 \cK_{n,t}
 =\{0\leq k\leq n: G_{n,t}(k)=k=G_{n,t}(k-1)\}.
\end{align*}
Roughly speaking, we think of $\cK_{n,t}(\omega)$ as the set of all $k$ such that $k/n=\rho_{n}(t)(\omega)$ for some $n$-player equilibrium $\rho_{n}(t)$. (This is not quite meaningful since equilibria can always be altered on nullsets.) More precisely, we have that if $\rho_{n}$ is a given equilibrium, then $n\rho_{n}(t)\in \cK_{n,t}$ a.s.\ by~(3.1). In particular, we will use below that $\{|x-\rho_{n}(t)|<\eps\}\subseteq \{\exists\, k\in \cK^{}_{n,t}:\, |x-\kn|<\eps\}$~a.s. for all $x\in[0,1]$ and $\eps>0$.
Finally, we also introduce the superset
\begin{align*}
 \cK^{*}_{n,t}
 &=\{0\leq k\leq n: G_{n,t}(k)=k\} \supseteq\cK_{n,t}
\end{align*}
which has no direct interpretation in terms of our game but is conveniently related to crossings of empirical distribution functions (see the proof below).

\begin{proposition}\label{pr:NairSheppKlass}
  Fix $t\geq0$ and let $x\in(0,1)$ satisfy $x+F_{t}(r-cx)=1$. Let $\alpha:=cf_{t}(r-cx)$ and assume that $\alpha>1$. Then
  $$
    \lim_{\eps\to0}\lim_{n\to\infty} P(\exists\, k\in \cK^{*}_{n,t}:\, |x-\kn|<\eps) = \frac{1-\theta}{\alpha-1}<1
  $$
  where $\theta\in(0,1)$ is defined through $\theta e^{-\theta}=\alpha e^{-\alpha}$.
\end{proposition}

\begin{proof}
  We first observe the local nature of the claim. Indeed, introducing the uniform random variables $U^{i}=F_{t}(Y_t^{i})$ we see that the event
  \begin{align*}
  A_{n,\eps}
  &=\{\exists\, k\in \cK^{*}_{n,t}:\, |x-\kn|<\eps\}\\
  &= \{ \exists\, 0\leq k\leq n:\, \#\{Y_t^i+c\smsp \tfrac{k}{n}\geq r\}=k,\;|x-\kn|<\eps\}\\
  &= \{ \exists\, 0\leq k\leq n:\, \#\{U^i\geq F_{t}(r-c\smsp \tfrac{k}{n})\}=k,\;|x-\kn|<\eps\} 
\end{align*}
  depends only on the values of $F_{t}$ in an $\eps$-neighborhood of $x$. In particular, for $\eps$ small enough, we may change $F_{t}$ outside that neighborhood to guarantee that the set of solutions of $u+F_{t}(r-cu)=1$ is $\{0,x,1\}$.

  Considering the c.d.f.\ $G(u)=1-F_{t}(r-cu)$, the proposition can be rephrased as the probability of having no crossings of the empirical distribution of $G$ and the (theoretical) uniform distribution near $x$:
  \begin{align*}
  A_{n,\eps}
  =\{ \exists\, t\in[0,1]:\, \tfrac{1}{n}\#\{G^{-1}(U^i)\leq t\}=t,\;|x-t|<\eps\}.
\end{align*}  
  (To see this identity, note that $\tfrac{1}{n}\#\{G^{-1}(U^i)\leq t\}=t$ implies  $t=k/n$ for some $0\leq k\leq n$.) Following~\cite{NairSheppKlass.86}, this problem can be related to boundary-crossing probabilities of Poisson processes which turn out to be computable. In particular, after changing $F_{t}$ as outlined above, the conditions of~\cite[Theorem~1]{NairSheppKlass.86} are satisfied for $G$ and noting that $\alpha=G'(x)$, this theorem yields the result.
\end{proof}

In view of $\cK_{n,t}\subseteq \cK^{*}_{n,t}$, we have the following consequence (see also Figure~\ref{fig:bounds}).
 
\begin{corollary}\label{co:decreasingLessOne}
  Fix $t\geq0$ and let $x\in[0,1]$ satisfy $x+F_{t}(r-cx)=1$. If~$x$ is strongly decreasing-transversal, then
  $$
    \lim_{\eps\to0}\lim_{n\to\infty} P(\exists\, k\in \cK^{}_{n,t}:\, |x-\kn|<\eps)<1.
  $$
\end{corollary}

\begin{figure}[bth]
  \centering
  \includegraphics[scale=.60]{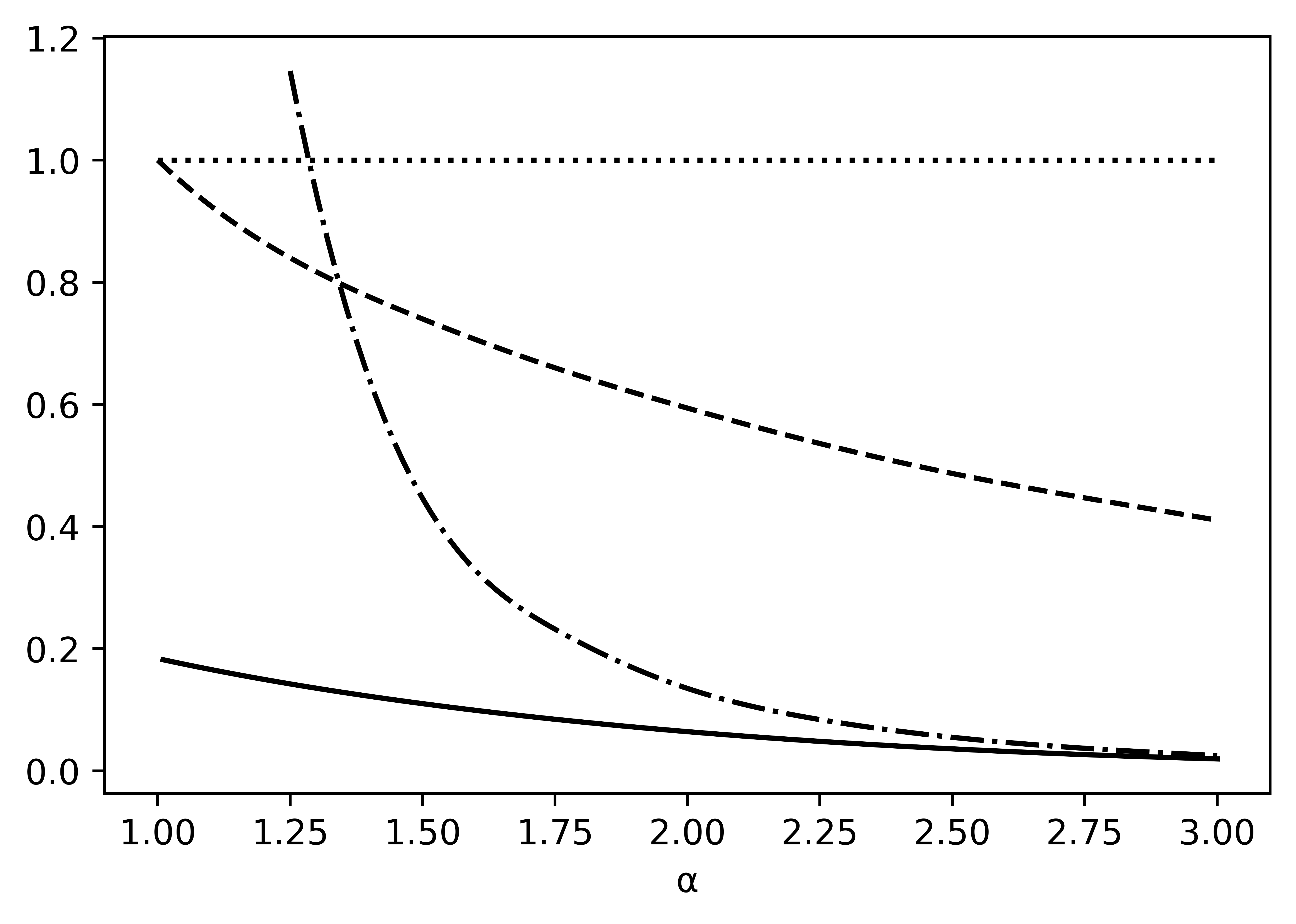}
\caption{Bounds for the probability of finding an $n$-player equilibrium near~$x$ as in Corollary~\ref{co:decreasingLessOne}. The dashed and dashed-dotted lines are the upper bounds derived from Proposition~\ref{pr:NairSheppKlass} and Proposition~\ref{pr:expectedNumber}, respectively. The solid line is the lower bound from Proposition~\ref{pr:limitProbLowerBound}.}
\label{fig:bounds}
\end{figure}

\begin{remark}\label{rk:weakerEquilibCond}
  One can ask if the non-existence result is related to the convention made in Section~\ref{se:nPlayer} that players do not consider their own impact on the state process. To address this question, we can drop the first equation in the equilibrium conditions~\eqref{eq:NplayerEquilibCond} and keep only the second (which seems uncontroversial); i.e., $\#\{Y_t^i+c\smsp \frac{k}{n}< r\}=n-k$.
  This corresponds to the definition of $\cK^{*}_{n,t}$ and Proposition~\ref{pr:NairSheppKlass} shows that non-existence holds even under this condition alone.
\end{remark}

\begin{remark}\label{rk:tangentialDecreasingCase}
  Heuristics suggest that in the tangential case of a decreasing-transversal $x$ with $\alpha=1$, the limiting probability is $1$; i.e., the equilibrium is in fact a limit of $n$-player equilibria. The tangential case is less important because it generically does not occur, in the same sense as discussed below Definition~\ref{de:H}.  We do not provide a rigorous result.
\end{remark}

In our last result, we determine the asymptotic expected number of equilibria close to $x$ (for both increasing- and decreasing-transversal cases). Importantly, it implies that this number is positive with positive probability. When $\alpha>1$ is not close to $1$, it also yields a fairly accurate upper bound for the probability of not finding an $n$-player equilibrium close to~$x$ (cf.~Example~\ref{ex:tent}) since the probability of finding more than one solution is small. On the other hand, we see that as $\alpha\to1$, the expected number of solutions tends to infinity, and in particular the probability of finding many solutions becomes large\footnote{In fact, one can show that $\lim_{\alpha\to1}\limsup_{n\to\infty} P(\#\{\cK_{n,t}\cap |x-\kn|<\eps\}=j)=0$ for all finite $j\geq0$ when $\eps>0$ is small enough.}.
\begin{proposition}\label{pr:expectedNumber}
  Fix $t\geq0$ and let $x\in(0,1)$ satisfy $x+F_{t}(r-cx)=1$. Let $\alpha:=cf_{t}(r-cx)$ and assume that $\alpha\neq1$. Then
  $$
    \lim_{\eps\to0}\lim_{n\to\infty} E[\#\{k\in \cK_{n,t}:\, |x-\kn|<\eps\}]= \frac{e^{-\alpha}}{|1-\alpha|}.
  $$
  In particular,
  $$
    \limsup_{\eps\to0}\limsup_{n\to\infty} P(\exists\, k\in \cK_{n,t}:\, |x-\kn|<\eps)\leq \frac{e^{-\alpha}}{|1-\alpha|}.
  $$
\end{proposition}

One consequence of Proposition~\ref{pr:expectedNumber} is that non-uniqueness is indeed the typical case for the $n$-player game, as claimed in the Introduction: under the stated smoothness assumption on $F_{t}$, we typically have at least one mean field equilibrium corresponding to $0\neq\alpha<1$ and then the proposition and Lemma~\ref{le:convToIncreasingStatic} imply that there is more than one $n$-player equilibrium, for large~$n$.

\begin{proof}[Proof of Proposition~\ref{pr:expectedNumber}.]
We may assume that $c=1$, and we drop the index $t$ everywhere. We denote 
$$
  \alpha(z)=f(r-z)
$$
and recall that $x\in(0,1)$ and $\alpha=\alpha(x)\neq1$. Fix $\eps>0$ and denote
$$
  x_{-}=x-\eps,\quad x_{+}=x+\eps,
$$
$$
  F_{-}=F(r-x-\eps),\quad F_{+}=F(r-x+\eps),
$$
$$
  \alpha_{-}=\inf_{|z-x|<\eps} \alpha(z),\quad \alpha_{+}=\sup_{|z-x|<\eps} \alpha(z),
$$
$$
  m(z)=\inf_{|z-x|\leq\eps} z(1-z),\quad M(z)=\sup_{|z-x|\leq\eps} z(1-z).
$$
We assume that $\eps$ is small enough such that $x_{\pm}\in(0,1)$ and $1\notin[\alpha_{-},\alpha_{+}]$.

\vspace{.5em}

\emph{Step 1: Bounds for $P(k\in\cK_{n})$.} Fix $n$ and let $U^{i}=F(Y^{i})$, $1\leq i\leq n$ so that $(U^{i})$ are i.i.d.\ $\Unif[0,1]$, and let $U^{(1)}\geq \dots\geq  U^{(n)}$ be the associated reverse order statistics. Noting that $U^{(k)}=U_{(n-k+1)}$ for the usual (increasing) order statistics $U_{(\cdot)}$, we have that $U^{(k)}\sim \Beta(n-k+1,k)$ and $U^{(k+1)}=U^{(k)}W_{k}^{\frac{1}{n-k}}$ where $W_{k}\sim \Unif[0,1]$ is independent; cf.\ \cite[Section~4]{ArnoldEtAl.13}.
Moreover, we note that $k\in\cK_{n}$ is equivalent to
\begin{equation}\label{eq:orderStat1}
  U^{(k)}\geq F(r-\konen)=:F_{k-1} \quad\mbox{and}\quad U^{(k+1)}\leq F(r-\kn)=:F_{k}.
\end{equation} 
As a result, for any deterministic integer $1\leq k\leq n$,
\begin{align}
  P(k\in\cK_{n}) 
  &= P\big(U^{(k+1)}\leq F_{k},\, U^{(k)}\geq F_{k-1}\big)\nonumber\\
  &= \int_{F_{k-1}}^{1} P\big(U^{(k+1)}\leq F_{k}\big|U^{(k)}=z\big)\, dP(U^{(k)}=z)\nonumber\\
  &= \int_{F_{k-1}}^{1} P\big(W\leq (F_{k}/z)^{n-k}\big|U^{(k)}=z\big)\, dP(U^{(k)}=z)\nonumber\\
  &= \frac{n!}{(n-k)!(k-1)!} \int_{F_{k-1}}^{1} F_{k}^{n-k}(1-z)^{k-1}\,dz\nonumber\\
  &= \binom{n}{k}F_{k}^{n-k}(1-F_{k-1})^{k} \label{eq:orderStat2}
\end{align}
where $dP(U^{(k)}=z)$ indicates integration with respect to the law of $U^{(k)}$. We may observe that this quantity is reminiscent of a binomial distribution except that the success probability changes with $k$.
Next, we use Taylor's theorem to find that
\begin{equation}\label{eq:alphaTaylor}
  F_{k-1}=F(r-\kn+\tfrac{1}{n})=F(r-\kn)+\alpha_{k}/n=F_{k}+\alpha_{k}/n
\end{equation}
where $\alpha_k=\alpha(\eta_k)$ with $\eta_k \in [\frac{k-1}{n},\frac{k}{n}]$ and in particular $\alpha_{k}\in[\alpha_{-},\alpha_{+}]$. Now suppose that $|x-\kn|<\eps$. Then $k\geq nx_{-}$, and using also $F_{k}\geq F_{-}$,
\begin{align*}
  P(k\in\cK_{n}) 
  &= \binom{n}{k}F_{k}^{n-k}(1-F_{k}-\alpha_{k}/n)^{k}\\
  &= \binom{n}{k}F_{k}^{n-k}(1-F_{k})^{k}\left(1-\frac{\alpha_{k}}{(1-F_{k})n}\right)^{k}\\
  &\leq \binom{n}{k}F_{k}^{n-k}(1-F_{k})^{k}\left(1-\frac{\alpha_{-}}{(1-F_{-})n}\right)^{nx_{-}}.
\end{align*}
The fact that
$$
  (1-y) \leq e^{-y} \leq (1-y)(1+o(y))
$$
as $y\to0$ applied with $y=w/n$ yields
$$
  (1-\tfrac{w}{n})^{n} \leq e^{-w} \leq (1-\tfrac{w}{n})^{n}\,(1+O(1/n))
$$
as $n\to\infty$, uniformly over $w$ in a compact interval. This leads us to the upper bound
\begin{align}\label{eq:proofStep1Conclusion}
  P(k\in\cK_{n}) 
  &\leq \binom{n}{k}F_{k}^{n-k}(1-F_{k})^{k} e^{-{\frac{\alpha_{-}x_{-}}{1-F_{-}}}}.
\end{align}
Similarly, we have the lower bound
\begin{align*}
  P(k\in\cK_{n})
  &\geq \binom{n}{k}F_{k}^{n-k}(1-F_{k})^{k}\left(1-\frac{\alpha_{+}}{(1-F_{+})n}\right)^{nx_{+}}\\
  &\geq \binom{n}{k}F_{k}^{n-k}(1-F_{k})^{k}e^{-{\frac{\alpha_{+}x_{+}}{1-F_{+}}}}\,(1+O(1/n)).
\end{align*}

\vspace{.5em}

\emph{Step~2: Decay away from $x$.} 
Let us recall Robbin's version~\cite{Robbins.55} of the Stirling approximation,
\begin{equation}\label{eq:robbins}
  \sqrt{2\pi n} (\tfrac{n}{e})^{n} e^{\tfrac{1}{12n+1}} \leq n! \leq \sqrt{2\pi n} (\tfrac{n}{e})^{n} e^{\tfrac{1}{12n}}, 
\end{equation} 
showing in particular that $n! = \sqrt{2\pi n}\, (\tfrac{n}{e})^{n}\,(1+O(1/n))$. 
Since $n-k$ and $k$ are comparable to $n$ when $|x-\kn|<\eps$, we have
$$
  \binom{n}{k}= (1+O(1/n)) \frac
  {\sqrt{2\pi n}\, (\tfrac{n}{e})^{n}}
  {\sqrt{2\pi (n-k)}\, (\tfrac{n-k}{e})^{(n-k)}\, \sqrt{2\pi k}\, (\tfrac{k}{e})^{k}}
$$
uniformly over all $k$ such that $|x-\kn|<\eps$.
This shows that
\begin{align}\label{eq:proofZsum}
    Z_{n,\eps}
    &:=\sum_{k:\, |x-\kn|<\eps} \binom{n}{k}F^{n-k}_{k}(1-F_{k})^{k}\nonumber\\
    & =(1+O(1/n))\sum_{k:\, |x-\kn|<\eps} \frac{1}{\sqrt{2\pi n(1-\kn)\kn}}\, 
    \frac{F^{n-k}_{k}(1-F_{k})^{k}}{(1-\kn)^{n-k}(\kn)^{k}}\nonumber\\
    &\leq \frac{1+O(1/n)}{\sqrt{m(x)}}\sum_{k:\, |x-\kn|<\eps} \frac{1}{\sqrt{2\pi n}}\, 
    \frac{F^{n-k}_{k}(1-F_{k})^{k}}{(1-\kn)^{n-k}(\kn)^{k}}.
\end{align}
Our next goal is to estimate the summand above. We introduce the function
$$
  \varphi(z)=(1-z)^{n-k}z^{k}
$$
so that 
\begin{equation}\label{eq:proofFraction}
\frac{F^{n-k}_{k}(1-F_{k})^{k}}{(1-\kn)^{n-k}(\kn)^{k}}=\frac{\varphi(1-F_{k})}{\varphi(\kn)}
\end{equation}
is the term in question.
We can use Taylor's theorem similarly as above to find
$$
  F_{k}=F(r-\kn)=F(r-x+x-\kn)=F(r-x)+\tilde{\alpha}_{k}(x-\kn)
$$
where $\tilde{\alpha}_{k}\in[\alpha_{-},\alpha_{+}]$. As $F(r-x)=1-x$, this equality can be rewritten as
$$
  F_{k}=1-\kn+(\tilde{\alpha}_{k}-1)(x-\kn).
$$
Introducing also
$$
  \psi(z)=\log\varphi(z)=(n-k)\log(1-z)+ k \log z,
$$
we have
$$
  \psi'(z) = -\frac{n-k}{1-z}+\frac{k}{z},\quad \psi''(z) = -n\left[\frac{1-\kn}{(1-z)^{2}}+\frac{\kn}{z^{2}}\right]<0
$$
and then $\psi'(k/n)=0$ shows that $\psi$ and $\varphi$ have a global maximum at~$k/n$. 
%
Taylor's theorem at the second order yields
$$
  \psi(1-F_{k})-\psi(\kn)=\psi(\kn-(\tilde{\alpha}_{k}-1)(x-\kn))-\psi(\kn)=\frac{\psi''(\xi_{k})}{2} (\tilde{\alpha}_{k}-1)^{2}(x-\kn)^{2}
$$
for a suitable number $\xi_{k}$ between $\kn$ and $\kn-(\tilde{\alpha}_{k}-1)(x-\kn)$. Therefore, we have 
$|\xi_{k}-x|<A\eps$, with $A=\max\{\alpha_+,1\}$. Using the above formula for $\psi''(z)$ and setting
$$
  \Gamma_{\eps}=\inf_{\substack{|p-x|<\eps\\|z-x|<A\eps}} \left[\frac{1-p}{(1-z)^{2}}+\frac{p}{z^{2}}\right],
$$
we arrive at
$$
  \psi(1-F_{k})-\psi(\kn)\leq -\frac{n}{2} \Gamma_{\eps} (\tilde{\alpha}_{k}-1)^{2}(x-\kn)^{2}.
$$
Setting also $\alpha_{*}=\alpha_{-}$ if $\alpha>1$ and $\alpha_{*}=\alpha_{+}$ if $\alpha<1$, exponentiating leads us to the desired estimate
$$
  \frac{\varphi(1-F_{k})}{\varphi(\kn)} \leq \exp \left(-\frac{n}{2} \Gamma_{\eps} (\alpha_{*}-1)^{2}(x-\kn)^{2}\right)
$$
and plugging this into~\eqref{eq:proofZsum} we have that 
\begin{align*}
    Z_{n,\eps} 
    &\le \frac{1+O(1/n)}{\sqrt{m(x)}}\sum_{k:\, |x-\kn|<\eps} \frac{1}{\sqrt{2\pi n}}\, \exp \left(-\frac{n}{2} \Gamma_{\eps} (\alpha_{*}-1)^{2}(x-\kn)^{2}\right).
\end{align*}
Set 
$$
  w_{k}= \sqrt{n\Gamma_{\eps}}|\alpha_{*}-1|(\kn-x)
$$
and note that
$$
  \Delta w: = w_{k}-w_{k-1}=\frac{1}{\sqrt{n}}\sqrt{\Gamma_{\eps}}|\alpha_{*}-1|.
$$
The above sum can then be written as 
\begin{align*}
    Z_{n,\eps} 
    & \le \frac{1+O(1/n)}{\sqrt{m(x)}}\sum_{k:\, |w_{k}|<\sqrt{n\Gamma_{\eps}}|\alpha_{*}-1|\eps} 
    \frac{1}{\sqrt{2\pi n}}\, e^{-w_{k}^{2}/2}\\
    &= \frac{1+O(1/n)}{\sqrt{m(x)\Gamma_{\eps}}}\frac{1}{|\alpha_{*}-1|}
    \sum_{k:\, |w_{k}|<\sqrt{n\Gamma_{\eps}}|\alpha_{*}-1|\eps} \frac{1}{\sqrt{2\pi}}\, e^{-w_{k}^{2}/2} \Delta w
\end{align*}
which suggests comparison with a Gaussian integral $\int_{\R}\frac{1}{\sqrt{2\pi}}\, e^{-w^{2}/2}\,dw=1$. Indeed, after subtracting the two largest summands neighboring the origin, the sum can be seen as a Riemann sum which is entirely below the integral. These two summands are $O(1/\sqrt{n})$ so that
$$
  \sum_{k:\, |w_{k}|<\sqrt{n\Gamma_{\eps}}|\alpha_{*}-1|\eps} \frac{1}{\sqrt{2\pi}}\, e^{-w_{k}^{2}/2} \Delta w \leq 1+O(1/\sqrt{n})
$$
and finally
\begin{align*}
    Z_{n,\eps} 
    & \leq \frac{1}{\sqrt{m(x)\Gamma_{\eps}}}\frac{1}{|\alpha_{*}-1|} \, (1+O(1/\sqrt{n})).
\end{align*}

\vspace{.5em}

\emph{Step 3: Conclusion.} Recalling~\eqref{eq:proofStep1Conclusion} we have
\begin{align*}
E[\#\{k\in \cK_{n}:\, |x-\kn|<\eps\}] 
& = \sum_{k:\, |x-\kn|<\eps} P(k\in\cK_{n}) \\
& \leq e^{-{\frac{\alpha_{-}x_{-}}{1-F_{-}}}} Z_{n,\eps}\\
&  \leq e^{-{\frac{\alpha_{-}x_{-}}{1-F_{-}}}} \frac{1}{\sqrt{m(x)\Gamma_{\eps}}}\frac{1}{|\alpha_{*}-1|} \, (1+O(1/\sqrt{n}))
\end{align*}
and hence 
\begin{align*}
\limsup_{n\to\infty} E[\#\{k\in \cK_{n}:\, |x-\kn|<\eps\}] 
& \leq e^{-{\frac{\alpha_{-}x_{-}}{1-F_{-}}}} \frac{1}{\sqrt{m(x)\Gamma_{\eps}}}\frac{1}{|\alpha_{*}-1|}.
\end{align*}
As $\eps\to0$ we have $x_{-}\to x$, $\alpha_{-}\to\alpha$, $\alpha_{*}\to\alpha$, $F_{-}\to F(r-x)=1-x$ and
$$
  m(x)\to x(1-x),\quad \Gamma_{\eps}\to \frac{1}{1-x}+\frac{1}{x}=\frac{1}{x(1-x)}.
$$
Thus,
\begin{align*}
\limsup_{\eps\to0}\limsup_{n\to\infty} E[\#\{k\in \cK_{n}:\, |x-\kn|<\eps\}] 
& \leq \frac{e^{-\alpha}}{|\alpha-1|}.
\end{align*}
The matching lower bound follows similarly after replacing $\alpha_{-}$ by $\alpha_{+}$, $F_{-}$ by $F_{+}$, and so on.
\end{proof}

\begin{remark}\label{rk:rootNconvergence}
  The above proof offers insight into the speed of convergence of $n$-player equilibria. Specifically, the estimates entail that if $\eps_{n}\downarrow0$ is such that $\eps_{n}\sqrt{n}\to \beta\in[0,\infty]$, then
  \begin{align*}
    E[\#\{k\in \cK_{n,t}:\, |x-\kn|<\eps_{n}\}] \to \frac{e^{-\alpha}}{|1-\alpha|} \mu\bigg(-\frac{|\alpha-1|}{\sqrt{x(1-x)}}\beta,\frac{|\alpha-1|}{\sqrt{x(1-x)}}\beta\bigg)
  \end{align*}
  where $\mu$ is the standard Gaussian distribution.
  Thus, a ball of radius $r_{n}/\sqrt{n}$ around $x$, where $r_{n}\to\infty$ arbitrarily slowly, will asymptotically contain all $n$-player equilibria converging to $x$, and this is optimal in the sense that if $\limsup r_{n}<\infty$ the ball will miss some solutions.
\end{remark}

In our final result we complement the upper bound in Proposition~\ref{pr:expectedNumber} by a lower bound. The gap between the bounds vanishes for large $\alpha$; see also Figure~\ref{fig:bounds}.

\begin{proposition}\label{pr:limitProbLowerBound}
  Fix $t\geq0$, let $x\in(0,1)$ satisfy $x+F_{t}(r-cx)=1$ and suppose that $\alpha:=cf_{t}(r-cx)>1$. Then
  $$
    \liminf_{\eps\to0}\liminf_{n\to\infty} P(\exists\, k\in \cK_{n,t}:\, |x-\kn|<\eps)\geq L(\alpha)>0
  $$  
  where
  $$
    L(\alpha) = \frac{e^{-\alpha}}{\big(\alpha-1\big)\left(1+2\sqrt{\frac{2}{|a_0|}}\left\{1-\Phi\left(\sqrt{2|a_0|}\right)\right\}\right)}
  $$
  with $a_0:=1-\alpha+\log(\alpha)<0$ and $\Phi$ is the standard normal c.d.f.
\end{proposition}

Since the lower bound is strictly positive, we can interpret the result as stating that~$x$ is necessarily part of a mixture which is itself a limit of $n$-player equilibria. In summary, when $x$ is strongly decreasing-transversal, we cannot find  $n$-player equilibria converging to $x$ at time~$t$, but at least we can find $n$-player equilibria converging to a randomized mean field equilibrium which charges~$x$.

\begin{proof}[Proof of Proposition~\ref{pr:limitProbLowerBound}]
We use the notation from the proof of Proposition~\ref{pr:expectedNumber} and suppress~$t$.
Let $\cK= \cK_{n,t}$  and $X=X_{n,\eps}=\#\{k\in \cK_{n,t}:\, |x-\frac{k}{n}|\le\eps\}$. Set $\mu=E[X]$ and let
$$
A=A_{n,\eps}=\{|X-c\mu|\ge c\mu\}
$$
for a constant $c>0$ to be chosen later. Clearly $P(X=0)\le P(A)$. Using the Markov inequality
$$
P(|X-c\mu|\ge c\mu)\le \frac{E\left((X-c\mu)^2\right)}{c^2 \mu^2}=
\frac{E\left(X^2\right)}{c^2 \mu^2}-\frac{2}{c}+1=
\frac{2}{c}\left(\frac{E\left(X^2\right)}{2c \mu^2}-1\right)+1
$$
and choosing $c=\theta\frac{E[X^2]}{2\mu^2}$ for some $\theta>1$, we obtain that
$$
P(X=0)\le1-\frac{4\mu^2(\theta-1)}{\theta^2E[X^2]}.
$$
Optimizing over the right-hand side, we note that $\theta=2$ yields the best bound, so we choose $c=\frac{E[X^2]}{\mu^2}$ and conclude that
\begin{equation}\label{eq:prooflimitProbLower1}
P(X=0)\le 1-\frac{\mu^2}{E[X^2]}= 1-\frac{E[X]^{2}}{E[X^2]}.
\end{equation}
Since we have already determined the limit of~$E[X]$ in Proposition~\ref{pr:expectedNumber}, our goal is to find an upper bound for~$E[X^2]$. 
To that end, we first compute
$$
P(k\in \cK, j\in \cK)=P(U^{(k+1)}\le F_k, U^{(k)}\ge F_{k-1},U^{(j+1)}\le F_j, U^{(j)}\ge F_{j-1})
$$
for $k<j$; recall the notation of~\eqref{eq:orderStat1}. In fact, this probability is zero for $j=k+1$, so we focus on $k+2\leq j$. Conditionally on $U^{(k+1)}=h< U^{(k)}=u$, the pair
$(U^{(j)},U^{(j+1)})$ has the same distribution as 
$(hV^{(j-(k+1))},hV^{(j-k)})$ where $V^{(\ell)}$ are the reverse order statistics of an i.i.d.\  sample $V_1,\cdots,V_{n-(k+1)}$ of size $n-(k+1)$ and distribution~$\Unif[0,1]$.
Thus, we have 
\begin{align}\label{eq:V1}
&P\left(U^{(j+1)}\le F_j, U^{(j)}\ge F_{j-1} \big| U^{(k+1)}=h, U^{(k)}=u\right)\nonumber\\
&=P\left(V^{j-(k+1)}\le \frac{F_j}{h}, V^{(j-k)}\ge \frac{F_{j-1}}{h}\right).
\end{align}
Clearly $P(V^{(j-k)}\ge \frac{F_{j-1}}{h})=0$ if $F_{j-1} \ge h$, so we only need to consider the case $h\in [F_{j-1},F_k]$. Using the formula developed in~\eqref{eq:orderStat2}, we obtain
\begin{align}\label{eq:V2}
&P\left(U^{(j+1)}\le F_j, U^{(j)}\ge F_{j-1} \big| U^{(k+1)}=h, U^{(k)}=u\right)\nonumber\\
&=\binom{n-(k+1)}{j-(k+1)} \left(\frac{F_j}h\right)^{n-j}\, \left(1-\frac{F_{j-1}}h\right)^{j-(k+1)}.
\end{align}
As above~\eqref{eq:orderStat1}, the joint density of $U^{(k)}$ and $U^{(k+1)}$ can be computed using the fact that $U^{(k)}\sim\Beta(n-k+1,k)$ and $U^{(k+1)}=W_{k}^{\frac{1}{n-k}}U^{(k)}$ where $W_{k}\sim\Unif[0,1]$ is independent of $U^{(k)}$:
$$
dP\left(U^{(k)}=u, U^{(k+1)}=h\right)=k(n-k)\binom{n}{k}(1-u)^{k-1} h^{n-(k+1)} \, \1_{0\le h\le u\le 1}\, du \, dh.
$$
Integrating with respect to this density and using the appropriate restrictions, we deduce that
\begin{align*}
P(k\in \cK, j\in \cK)
&=k(n-k)\binom{n}{k} \binom{n-(k+1)}{j-(k+1)} F_j^{n-j}  \\
&\quad\,\times\int_{F_{k-1}}^1 (1-u)^{k-1} \, du \int_{F_{j-1}}^{F_k} (h-F_{j-1})^{j-(k+1)}\, dh\\
&=\binom{n}{k} (1-F_{k-1})^k \, 
\frac{n-k}{j-k}\binom{n-(k+1)}{j-(k+1)} F_j^{n-j} (F_k-F_{j-1})^{j-k}\\
&=\binom{n}{k}(1-F_{k-1})^k F_k^{n-k} \, \binom{n-k}{j-k} 
\left(\frac{F_j}{F_k}\right)^{n-j} \left(1-\frac{F_{j-1}}{F_k}\right)^{j-k}\\
&\le \binom{n}{k}(1-F_{k-1})^k F_k^{n-k} \, \binom{n-k}{j-k} 
\left(\frac{F_j}{F_k}\right)^{n-j} \left(1-\frac{F_{j}}{F_k}\right)^{j-k}.
\end{align*}
By a repeated application of~\eqref{eq:alphaTaylor} we have that $\frac{F_{j}}{F_k}=1-\frac{\alpha_j (j-k)}{n F_k}$ 
for some $\alpha_j \in [\alpha_-,\alpha_+]$ and hence the last two terms above satisfy
\begin{align*}
\left(\frac{F_j}{F_k}\right)^{n-j} \left(1-\frac{F_{j}}{F_k}\right)^{j-k}
&\leq \left[1-\frac{\alpha_j(j-k)}{n F_k}\right]^{n-j} \left[\frac{\alpha_j(j-k)}{n F_k}\right]^{j-k}\\
&\leq \exp\left(-\alpha_-(j-k)\frac{n-j}{nF_k}\right) (\alpha_+)^{j-k} \left(\frac{j-k}{nF_k}\right)^{j-k}\\
&\leq \exp\left(-\alpha_-(j-k)\frac{1-x_+}{F_+}\right) (\alpha_+)^{j-k} \left(\frac{j-k}{nF_k}\right)^{j-k}.
\end{align*}
On the other hand, Stirling's approximation as in~\eqref{eq:robbins} yields
\begin{align*}
&\binom{n-k}{j-k} \left(\frac{j-k}{nF_k}\right)^{j-k} \!\!\!\!\!\!
= \frac{(n-k)!}{(n-j)! (j-k)! }\left(\frac{j-k}{nF_k}\right)^{j-k}\!\!\!\!\!
\le \left(\frac{n-k}{nF_k}\right)^{j-k}\! \frac{(j-k)^{j-k}}{(j-k)! }\\
&\le \left(\frac{1-x_-}{F_-}\right)^{j-k}\!\!\!\!\!(j-k)^{j-k}\!\left[\left(\frac{(j-k)}{e}\right)^{j-k}\!\!\!\!\!\sqrt{2\pi(j-k)} \exp\left(\frac1{12(j-k)+1}\right)\right]^{-1}\\
&\le \left(\frac{1-x_-}{F_-}\right)^{j-k}\frac{e^{j-k}}{\sqrt{2\pi(j-k)}}\,.
\end{align*}
As a result, we obtain the upper bound
\begin{equation}\label{eq:combEstimate}
P(k\in \cK, j\in \cK)\le \binom{n}{k}(1-F_{k-1})^k F_k^{n-k} \frac{1}{\sqrt{2\pi(j-k)}}
\exp(a(j-k))
\end{equation}
where
$$
a=a(\alpha,\eps):=1-\alpha_-\frac{1-x_+}{F_+}+\log(\alpha_+)+\log\left(\frac{1-x_-}{F_-}\right).
$$
If the following sums run over indices $i$ with $|x-i/n|\le \eps$, we can express the second moment of $X$ as
\begin{align*}
  E[X^{2}] 
  & =  E\left[\left(\sum\nolimits_{k}\1_{k\in\cK}\right)\left(\sum\nolimits_{j}\1_{j\in\cK}\right)\right] 
   = E\left[\sum\nolimits_{k,j}\1_{k\in\cK}\1_{j\in\cK}\right] \\
  & =  E\left[\sum\nolimits_{k}\1_{k\in\cK} + 2\sum\nolimits_{k<j}\1_{k\in\cK}\1_{j\in\cK}\right]\\
  &=  \sum\nolimits_{k} P(k\in\cK) + 2\sum\nolimits_{k<j} P(k\in\cK,\,j\in\cK).
\end{align*} 
Thus, \eqref{eq:combEstimate} leads to 
\begin{align*}
E[X^2]
&=\sum_{k:\, |x-k/n|\le \eps} P(k\in \cK)+
2\sum_{\substack{k,j:\,  j\ge k+2, \\ |x-k/n|\le \eps, \\ |x-j/n|\le \eps}} P(k\in \cK, j\in \cK)\nonumber\\ 
&\le E[X]+ \frac{2}{\sqrt{2\pi}}\, E[X] \sum\limits_{\ell=2}^{n(x_+-x_-)} \frac{1}{\sqrt{\ell}} \, e^{a\ell}.
\end{align*}
Note that $a_{0}:=\lim_{\eps\downarrow 0}a(\alpha,\eps)=1-\alpha+\log(\alpha)$ is strictly negative since $\alpha>1$. Thus, $a=a(\alpha,\eps)<0$ for $\eps$ small enough, so that $\frac{1}{\sqrt{\ell}} \, e^{a\ell}$ is summable.
More precisely,
$$
\frac{1}{\sqrt{2\pi}}  \sum_{\ell=2}^\infty \frac{1}{\sqrt{\ell}} \, e^{a\ell}\le
\frac{1}{\sqrt{2\pi}} \int_1^\infty \frac{1}{\sqrt{\ell}} e^{a\ell} \, d\ell=\sqrt{\frac{2}{|a|}}
\frac{1}{\sqrt{2\pi}} \int_{\sqrt{2|a|}}^\infty e^{\frac{-z^2}{2}} \, dz.
$$
Recalling also that $\lim_{\eps \to 0}\lim_{n\to \infty} E[X]=\frac{e^{-\alpha}}{|\alpha-1|}=:H(\alpha)$ by Proposition~\ref{pr:expectedNumber}, we deduce that
\begin{align*}
\limsup_{\eps \to 0}\limsup_{n\to \infty} E[X^2]
\le H(\alpha)\left(1+2\sqrt{\frac{2}{|a_0|}}\left(1-\Phi\left(\sqrt{2|a_0|}\right)\right)\right)
\end{align*}
and combining this with~\eqref{eq:prooflimitProbLower1} yields the claim.
\end{proof}


\bibliography{stochfin}
\bibliographystyle{plain}

\end{document}